\documentclass[a4paper,10pt]{article}
\usepackage{graphicx,amsmath,amsfonts,amssymb,amsthm,color,
framed, placeins %,stmaryrd
}

\usepackage{bbm}
\usepackage[font={small,it}]{caption}

\newcommand{\dint}{\mathrm{d}}

\newtheorem{thm}{Theorem}[section]

\newtheorem{prop}[thm]{Proposition}
\newtheorem{assumption}{Assumption}[section]
\newtheorem{rem}[thm]{Remark}
\numberwithin{equation}{section}

\title{Exact simulation of the first-passage time of diffusions}
\begin{document}
\author{S. Herrmann$^1$ and C. Zucca$^2$\\[5pt]
\small{$^1$Institut de Math{\'e}matiques de Bourgogne (IMB) - UMR 5584, CNRS,}\\
\small{Universit{\'e} de Bourgogne Franche-Comt\'e, F-21000 Dijon, France} \\
\small{Samuel.Herrmann@u-bourgogne.fr}\\[5pt]
\small {$^2$Department of Mathematics 'G. Peano', }\\
\small{University of Torino, Via Carlo Alberto 10,
10123 Turin, Italy,}\\ 
\small{cristina.zucca@unito.it}
}

 \date{\today}
\maketitle

\begin{abstract}  Since diffusion processes arise in so many different fields, efficient technics for the simulation of sample paths, like discretization schemes, represent crucial tools in applied probability. Such methods permit to obtain approximations of the first-passage times as a by-product. For efficiency reasons, it is particularly challenging to simulate directly this hitting time by avoiding to construct the whole paths. In the Brownian case, the distribution of the first-passage time is explicitly known and can be easily used for simulation purposes. The authors introduce a new rejection sampling algorithm which permits to perform an exact simulation of the first-passage time for general one-dimensional diffusion processes. The efficiency of the method, which is essentially based on Girsanov's transformation, is described through theoretical results and numerical examples.
\end{abstract}
\par\medskip

\par\medskip

\section*{Introduction} 
First-passage times of diffusion processes appear in several fields, i.e. economics \cite{Hu2012}, mathematical finance \cite{Janssen, Linetsky}, neuroscience \cite{Burkitt-2006,Sacerdote-2013}, physics \cite{Redner}, psychology \cite{Navarro} and reliability theory \cite{Pieper}.
In neuroscience particular diffusion processes model the evolution of the membrane potential and the first-passage times of the considered process throught a suitable boundary describes the neuronal firing times (cf. \cite{Burkitt-2006,Sacerdote-2013}). Likewise, in metrology or in quality control, diffusion processes are used to describe the error of some instrument which is required to be bounded. The first-passage time through a suitable boundary reproduces the first time the error becomes out of control \cite{ZucTav}. In this setting, rapid detection of anomalies corresponds to recognize the optimal stopping time of a diffusion process \cite{ZucTavPes}. The importance of practical applications dealing with the first passage time motivated many mathematical studies on the subject. 

To introduce the corresponding mathematical problem, let us consider the stochastic process $(X_t,\ t\ge 0)$ the solution of the SDE:
\[
dX_t=b(X_t)dt+\sigma(X_t)dB_t,\quad X_0=x\in\mathbb{R},
\]
where $(B_t,\ t\ge 0)$ stands for the standard one-dimensional Brownian motion. We denote by $\tau_L$ the first time the diffusion reaches the level $L$:
\begin{equation}\label{eq:def:tau}
\tau_L:=\inf\{t\ge 0: \ X_t=L  \}.
\end{equation}
The study of the hitting times $\tau_L$  and of its approximations for general diffusion processes is an active area of research.

Explicit expressions or even asymptotic solutions to boundary crossing problems are only available in rare situations which moreover do not correspond to interesting instances for applications. Consequently
%As the cases when closed form solutions or asymptotic solutions are available
%are rare and do not correspond to instances of interest for applications, 
studies on this subject have been directed to the development of alternative methods like numerical or simulative ones. 

%The aim of this paper is to emphasize an efficient method for simulating the first-passage time of diffusions. In particular,

As far as numerical methods are concerned, they are efficient but need specific algorithms for each function of the hitting time $\tau_L$. 
Most of the studies focus on  the probability density function $p$ defined by $p(t)dt=\mathbb{P}(\tau_L\in dt)$
which satisfies Volterra-type integral equations. It suffices therefore to approximate certain integral. %or differential equations by a series expansion method. 
Examples of studies  in this direction are Durbin \cite{Durbin-1985,Durbin-1992}, Ferebee \cite{Ferebee-1983}, Ricciardi et al. \cite{Ricciardi-al-1984}; Giorno et al. \cite{Giorno-al-1989},
Sacerdote and Tomassetti \cite{Sacerdote-Tomassetti-1996}. The numerical approach proposed in \cite{Buonocore-1987} seems to be particularly efficient. This type of approach is also used in \cite{Benedetto2013, SacTambZuc} to deal with first passage times of a two-dimensional correlated diffusion process.
For particular families of diffusion which can be constructed using simple transformations of a Brownian motion, P\"{o}tzelberger and Wang \cite{Wang-Potzelberger-1997,Wang-Potzelberger-2001,Wang-Potzelberger-2007} proposed a different approach based on the explicit Brownian crossing probabilities. 

The second important course of action uses simulations to determine the hitting times $\tau_L$. These methods often require a strong computational effort and are affected by a non negligible approximation error. Since stopping times for time discrete Markov processes can be quite easily simulated, a natural way to deal with diffusion processes is to introduce a discretization scheme for the corresponding stochastic differential equation. However, only an upper bound of the stopping time can be determined through this approximation. It is therefore  important to improve the algorithm. In \cite{Broadie-Glasserman-Kou-1997} and \cite{Gobet-Menozzi-10} a shift of the boundary is proposed to improve the approximation, while in \cite{Giraudo-Sacerdote-1998,Giraudo-Sacerdote-1999, Giraudo-al-2001} the proposed method compute on each small interval of the time discretization, the probability for the process to cross the boundary in this time window. Such a method needs precise asymptotics of hitting probabilities for pinned diffusions.  Gobet \cite{Gobet-2000} exteded the study to a multidimensional diffusion and described the error for both the discrete and the continuous Euler schemes. %first proposed a way to freeze the coefficients of the diffusion on each small time interval in order to obtain precise asymptotics of the crossing probabilities. 
%Nevertheless it has been pointed out, by the numerical treatment of some precise examples, that such approximations are not satisfactory: see Giraudo and Sacerdote \cite{Giraudo-Sacerdote-1999} (O.U. process and Feller model), who also suggest some formulas for the computation of the crossing probability, see also \cite{Giraudo-al-2001}. 
Precise asymptotics for general pinned diffusions are pointed out by Baldi and Caramellino \cite{Baldi-Caramellino-2002} permitting to deal with a larger class of diffusions. 

\noindent Let us just notice that besides the simulated approximation methods there also exist simulated exact methods dealing with the sample of diffusion paths on a fixed time interval, 
%where the main challenge is to control and reduce the approximation error as far as possible. A pillar in this direction is the approach introduced by 
 first introduced by Beskos \& Roberts \cite{beskos2005exact}. Several modifications of this algorithm have been proposed \cite{Beskos-2006, Beskos-2008, Jenkins}.
However, the two main approaches are not exclusive: recently Herrmann and Tanr\'e \cite{Herrmann-Tanre-2016} proposed a new simulation method which  has no link with the two approaches just described. It is based on an iterative algorithm and leads to an efficient approximation of Brownian hitting times for curved boundaries, and, by extension, can deal with diffusion processes expressed as functionals of a Brownian motion.

This paper is a contribution in the framework of the simulated exact methods. Here we present an efficient method for simulating directly the first-passage time $\tau_L$ without building the whole paths of the diffusion process and without any approximation error. The method can be summarized as follows: we need an algorithm producing a random variable $Y$ such that $Y$ and $\tau_L$ are identically distributed. In particular, this approach permits to reduce the approximation error of $\mathbb{E}[\psi(\tau_L)]$, where $\psi:\mathbb{R}\to \mathbb{R}$ is any measurable and bounded function, to the only Monte-Carlo one. Our approach rests on two important facts: it is easy to simulate the first-passage time of the standard Brownian motion on one hand and Girsanov's transformation allows to link the Brownian paths to any diffusion paths on the other hand. Combining these features, we highlight the following  acceptance-rejection algorithm (A0):
\begin{align*}
\begin{array}{l}
\mbox{Step 1. Simulate}\ T_L \ \mbox{the first hitting time of the Brownian motion} \\
\mbox{Step 2. If the event } A\  \mbox{is true, then} \ Y= T_L\ \mbox{otherwise go to Step 1.} 
\end{array}
\end{align*} 
The first challenge is to find an appropriate event $A$. Subsequently we present alternatives in order to reduce the averaged number of steps of this acceptance/rejection algorithm. Our main idea originates in two different mathematical studies; let us now focus our attention on them.
\begin{itemize}
\item Ichiba \& Kardaras \cite{ichiba2011efficient} proposed an efficient estimation of the first-passage density using on one hand Girsanov's transformation and on the second hand a Monte Carlo procedure in order to estimate the Radon-Nikodym derivative which involves a Bessel bridge of dimension $\delta=3$. This interesting approach permits to deal with a large family of diffusion processes but introduces a Monte-Carlo approximation error.
\item   The second pillar is the procedure of exact simulation introduced by Beskos \& Roberts \cite{beskos2005exact}. Using a judicious rejection sampling algorithm, they simulate in a surprisingly simple way a diffusion process on any finite time interval. They propose also to use their algorithm in order to simulate extremes and hitting times by building a skeleton of the whole trajectory.  
\end{itemize}
The algorithm proposed here, based on these two pillars,  outperform \cite{beskos2005exact} when the objective is the simulation of $\tau_L$ instead of the whole diffusion paths.\\

The material is organized as follows: after reminding the famous Girsanov transformation, we present in details the rejection sampling algorithm in Section \ref{sec:1} and prove that the outcome $Y$ has the expected distribution. In Section \ref{sec:efficiency}, we focus our attention on the efficiency of the procedure: the number of steps is described and upper-bounds are computed. Unfortunately our procedure  can not be applied to any diffusion processes since it requests different technical assumptions.  We discuss in Section \ref{sec:before} and Section \ref{sec:unbounded} two modifications of the main algorithm in order to weaken these assumptions. Finally, for completeness, numerical examples permit to appreciate and illustrate the efficiency of the method (Section \ref{sec:numerics}).

\section{Rejection sampling method for first-passage time of diffusions}\label{sec:1}
Let us first consider $(X_t,\ t\ge 0)$ the solution of the following stochastic differential equation:
\begin{equation}\label{eq:0}
dX_t=b(X_t)dt+\sigma(X_t)dB_t,\quad X_0=x\in\mathbb{R},
\end{equation}
where $(B_t,\ t\ge 0)$ stands for the standard one-dimensional Brownian motion. We denote by $\tau_L$ the first time the diffusion reaches the level $L$ as defined in \eqref{eq:def:tau}. It suffices to assume the existence of a unique weak solution to the equation, see for instance \cite{K-S} for the corresponding conditions.  In order to simplify the presentation of all algorithms, we restrict our study to the constant diffusion case: $\sigma(x)\equiv 1$. Using the \emph{Lamperti transform},  we observe that this restriction is not sharp at all. Indeed if $X$ is solution to \eqref{eq:0} then by fixing $x\in\mathbb{R}$, we define
\[
Y_t=\eta(X_t):=\int_x^{X_t}\frac{1}{\sigma(u)}\,\dint u
\] 
which satisfies therefore the equation: $dY_t=b_0(Y_t)\,dt+dB_t$ with \[
b_0(x)=\frac{b\circ\eta^{-1}(x)}{\sigma\circ \eta^{-1}(x)}-\frac{1}{2}\, \sigma'\circ\eta^{-1}(x).
\]
That is why, from now on, $X$ stands for the unique solution of 
\begin{equation}
\label{eq:simple-sde}
dX_t=b(X_t)dt+dB_t, \quad t\ge 0, \quad X_0=x.
\end{equation}
and the level to reach satisfies $L>x$. Let us note that, for the particular Brownian case: $b(x)\equiv 0$, the first-passage time is extremely simple to simulate since $\tau_L$ and $(L-x)^2/G^2$ have the same distribution where $G$ is a standard gaussian random variable.  The aim is to use this simple form when considering the general diffusion case and the important tool for such a strategy is Girsanov's formula. We first recall it before introducing the main rejection algorithm.

\subsection{Girsanov's transformation}
 Let us first introduce $3$ important expressions related to the drift term $b$ : 
\begin{equation}\label{eq:def:3}
\beta(x)=\int_0^x b(y)\dint y,\quad  p(x)=\int_0^x e^{-\beta(y)}\,\dint y\quad \mbox{and} \quad \gamma:=\frac{b^2+b'}{2}. 
\end{equation}
\begin{assumption} \label{assu-1} The drift term $b\in\mathcal{C}^1(]-\infty, L])$. %and $\displaystyle\lim_{x\to-\infty}p(x)=-\infty$. 
\end{assumption}
Grisanov's transformation permits to link the diffusion process $X$ solution of \eqref{eq:simple-sde} to the Brownian motion (a slight modification of Proposition 2.1 in \cite{ichiba2011efficient}).
\begin{prop} \label{prop:Girsanov} Under Assumption \ref{assu-1}, for any bounded  measurable function $\psi:\mathbb{R}\to\mathbb{R}$, we obtain
\begin{equation}\label{prop:eq:1}
\mathbb{E}_\mathbb{P}[\psi(\tau_L)1_{\{ \tau_L<\infty \}}]=\mathbb{E}_\mathbb{Q}[\psi(\tau_L)\eta(\tau_L)]\exp \Big\{\beta(L)-\beta(x)\Big\},
\end{equation}
where $\mathbb{P}$ (resp. $\mathbb{Q}$) corresponds to the distribution of the diffusion $X$ (resp. the Brownian motion $B$) and
\begin{equation}
\label{prop:eq:2}
\eta(t):=\mathbb{E}\Big[ \exp - \int_0^t \gamma(L-R_s)\dint s\Big | R_t=L-x\Big].
\end{equation}
Here $(R_t, \, t\ge 0)$ stands for a $3$-dimensional Bessel process with $R_0=0$.
\end{prop}
\begin{proof}%[Sketch of proof] 
Let us denote $\zeta$ the explosion time of the diffusion $X$ defined by \eqref{eq:simple-sde}. The Feller test for explosion (see for instance Proposition 5.22 in \cite{K-S}) claims that Assumption \ref{assu-1} leads to $\tau_L<\zeta$, $\mathbb{P}$-almost surely. By applying successively Girsanov's transformation (Proposition 1.7.5.4 in \cite{Jeanblanc-2009}) and  It\^o's formula, we obtain
\begin{align*}
\mathbb{E}_\mathbb{P}[\psi(\tau_L)1_{\{ \tau_L<\infty  \}}]&=\mathbb{E}_\mathbb{Q}\Big[ \psi(\tau_L)\exp\left( \int_0^{\tau_L} b(B_s)\dint B_s-\frac{1}{2}\int_0^{\tau_L} b^2(B_s)\dint s \right) \Big]\\
&=\mathbb{E}_\mathbb{Q}\Big[ \psi(\tau_L)\exp\left( \beta(B_{\tau_L})-\beta(x)-\int_0^{\tau_L} \gamma(B_s)\dint s \right) \Big].
\end{align*}
Under $\mathbb{Q}$, $(B_t)$ corresponds to a one-dimensional Brownian motion with starting point $x$ and $\tau_L=\inf\{t\ge 0:\ B_t=L  \}<\infty$ a.s.  (recurrence of the Brownian paths). Hence $B_{\tau_L}=L$ and consequently
\begin{align*}
\mathbb{E}_\mathbb{P}[\psi(\tau_L)1_{\{ \tau_L<\infty  \}}]&=e^{\beta(L)-\beta(x)}\mathbb{E}_\mathbb{Q}\Big[ \psi(\tau_L)\exp -\int_0^{\tau_L} \gamma(B_s)\dint s\Big]\\
&=e^{\beta(L)-\beta(x)}\mathbb{E}_\mathbb{Q}\Big[ \psi(\tau_L)\eta(\tau_L)\Big],
\end{align*}
where 
\[
\eta(t):=\mathbb{E}_{\mathbb{Q}}\Big[\exp -\int_0^{t} \gamma(B_s)\dint s\Big| \tau_L=t\Big].
\]
In order to conclude it suffices to note that, given $\{ \tau_L=t \}$, $R_s:= L-B_{t-s}$ is a Bessel bridge \cite{Williams-1974} for $s\in[0,t]$.
\end{proof}

%Let us note that, under the probability $\mathbb{Q}$, the FPT $\tau_L$ is easy to simulate. Indeed the Brownian first passage time of level $L$ has the same distribution as the random variable $L^2/G^2$ where $G$ is a standard Gaussian r.v. The simulation therefore becomes straightforward providing that the function $\eta$ is well-known. This is not the case in practice, that is why we need an algorithm in order to deal with $\eta$.

\subsection{Description of the algorithm}
The rejection sampling algorithm for the simulation of $\tau_L$ is based on Proposition \ref{prop:Girsanov}. We  need to introduce the following 
\begin{assumption}\label{assu-2}
The function $\gamma$ defined by \eqref{eq:def:3} is non-negative and the first-passage time  satisfies $\tau_L<\infty$ a.s.
\end{assumption}
These main conditions permit the function $\eta$ defined by \eqref{prop:eq:2} to belong to the interval $[0,1]$ and therefore  to represent a probability of rejection in the procedure. 
\begin{rem}\label{rem}
Let us just note that $\tau_L<\infty$ a.s. if the drift term satisfies one of the additional conditions: either $\displaystyle\lim_{x\to-\infty}p(x)=-\infty$ (see \eqref{eq:def:3} for the definition of $p$) or $\displaystyle \lim_{x\to-\infty}v(x)<\infty$ where
\begin{equation}\label{eq:cond:add} 
v(x)=\int_0^x\int_0^y \frac{2\,p'(y)}{p'(z)} \ \dint z\, \dint y,
\end{equation}
according to the Feller test for explosion (Proposition 5.5.32 in \cite{K-S}).
\end{rem}
Let us now focus our attention to a first theoretical description of the acceptance-rejection algorithm.

\begin{framed}\emph{
\noindent {\bf\sc Algorithm (A0).} 
\\[3pt]
Let us fix the level $L>x$.\\
{\bf\rm  Step 1:} Simulate a non-negative random variable $T$ with p.d.f. $f_T$. \\[3pt]
{\bf\rm Step 2:} Simulate a $3$-dimensional Bessel process $(R_t)$ on the time interval $[0,T]$ with endpoint $R_T=L-x$. We define by $D_{R,T}$ the stochastic domain:
\[
D_{R,T}:=\Big\{(t,v)\in[0,T]\times\mathbb{R}_+:\ v\le \gamma(L-R_t)  \Big\}.
\]
This domain depends on both random elements $(R_t)$ and $T$.\\[3pt]
{\bf\rm Step 3:} Simulate a Poisson point process $N$ on the state space $[0,T]\times\mathbb{R}_+$, independent of the Bessel process,  whose intensity measure is the Lebesgue one.\\[3pt]
{\bf\rm Step 4:} If $N(D_{R,T})=0$ then set $Y=T$ otherwise go to Step 1.\\[5pt]
{\bf\rm Outcome:} the random variable $Y$.}
\end{framed}  
We shall discuss later on, how this algorithm can be applied in practice.

\begin{thm} 
\label{thm:dens}
If $\gamma$ is a non-negative function then the \emph{p.d.f.} $f_Y$ of the outcome variable $Y$ satisfies
\begin{equation}
\label{eq:thm:dens}
f_Y(t)=\frac{1}{\Xi}\,\eta(t)f_T(t), 
\end{equation}
where $\eta$ is given by \eqref{prop:eq:2} and $\Xi$ stands for the normalization coefficient
\[
\Xi:=\int_0^\infty \eta(t)f_T(t)\dint t.
\]
In particular, under Assumption \ref{assu-1} \& \ref{assu-2}, if $T/(L-x)^2$ has the same distribution as $1/G^2$ where $G$ is a standard Gaussian r.v., then $Y$ and $\tau_L$, defined by \eqref{eq:def:tau} \& \eqref{eq:simple-sde}, are identically distributed.
\end{thm}
Let us observe that in the particular case $T\sim\frac{(L-x)^2}{G^2}$, \eqref{prop:eq:1} leads to the following identify:
\[
\Xi=\exp-\Big\{\beta(L)-\beta(x)\Big\}.
\]
\begin{proof}
The arguments presented here are quite classical for a rejection sampling method. We introduce $\mathcal{I}$ the number of iterations in Algorithm (A0). For each iteration $\mathcal{I}=i$, we denote by $T^{(i)}$, $R^{(i)}_t$ and $N^{(i)}$ the generated random variables or processes. Of course each iteration corresponds to independent variables. Let $\psi$ be a measurable bounded function. We get
\begin{align*}
\mathbb{E}[\psi(Y)]&=\sum_{i=1}^\infty \mathbb{E}[\psi(Y)1_{\{ \mathcal{I}=i \}}]=\sum_{i=1}^\infty \mathbb{E}\Big[ \psi(T^{(i)})1_{E_1\cap \ldots\cap E_{i-1}\cap \overline{E}_i} \Big],
\end{align*}
where $E_{i}:=\Big\{N^{(i)}(D_{R^{(i)},T^{(i)}})\neq 0\Big\}$ and $\overline{E}_i=\Omega\setminus E_i$. Hence, using the independence property, we obtain
\begin{align*}
\mathbb{E}[\psi(Y)]&=\sum_{i=1}^\infty \mathbb{E}\Big[ \psi(T^{(i)})1_{\overline{E}_i}\Big]\mathbb{P}(E_1)\ldots\mathbb{P}(E_{i-1})\\
&=\sum_{i=1}^\infty \mathbb{E}\Big[ \psi(T^{(1)})1_{\overline{E}_1}\Big]\mathbb{P}(E_1)^{i-1}=\mathbb{E}\Big[ \psi(T^{(1)})1_{\overline{E}_1}\Big](1-\mathbb{P}(E_1))^{-1}.
\end{align*}
From now on, we delete the index for notational simplicity. Therefore
\begin{align*}
\mathbb{E}[\psi(Y)]&=\frac{1}{\mathbb{P}(N(D_{R,T})=0)}\ \mathbb{E}\Big[ \psi(T)\mathbb{P}(N(D_{R,T})=0|T)\Big].
%&=\mathbb{E}\Big[ \psi(T^{(1)})\frac{1_{\overline{E}_1}}{\mathbb{P}(\overline{E}_1)}\Big].
\end{align*}
The distribution of the Poisson point process leads to:
\begin{align*}
\mathbb{P}(N(D_{R,T})=0|T)&=\mathbb{E}\Big[\mathbb{P}(N(D_{R,T})=0|R,T)\Big| T\Big]=\mathbb{E}\Big[\exp-\lambda(D_{R,T})\Big| T\Big]\\
&=\mathbb{E}\Big[\exp-\int_0^T\gamma(L-R_t)\dint t\Big| R_T=L-x, T\Big]=\eta(T).
\end{align*}
We deduce 
\[
\mathbb{E}[\psi(Y)]=\frac{1}{\mathbb{P}(N(D_{R,T})=0)}\ \mathbb{E}[\psi(T)\eta(T)].
\]
For $\psi\equiv 1$, we obtain $\mathbb{P}(N(D_{R,T})=0)=\mathbb{E}[\eta(T)]$ which implies the statement of the proposition.
\end{proof}

Algorithm (A0) seems quite difficult to achieve in practice since we do not know how to simulate a Poisson point process on the state space $[0,T]\times\mathbb{R}_+$ on one hand and since the domain $D_{R,T}$ depends on the whole trajectory of $(R_t)$ on the other hand. In order to overcome the first difficulty, we reduce the domain to a bounded one introducing the following assumption:
\begin{assumption}\label{assu-3}
there exists a constant $\kappa>0$ such that 
\begin{equation}\label{hyp}
0\le \gamma(x)\le \kappa,\quad \mbox{for all }\ x\in ]-\infty,L]. 
\end{equation}
\end{assumption}
It suffices therefore to simulate the Poisson point process on the reduced space $[0,T]\times[0,\kappa]$. The second difficulty is not insurmontable. In fact we don't actually need to precisely describe the entirely domain $D_{R,T}$: we only need to know if the points of the Poisson process belong to it. This can fortunately be obtained in a quite simple way, that's why we propose two modifications of the theoretical Algorithm (A0) -- Algorithm (A1) and (A2) -- which are easy to implement. These modifications essentially concern the way used to simulate the Poisson point process. %Instead of generating a Poisson random variable (Step 3), it is possible to deal with exponentially distributed random variables.
%\begin{framed}\emph{
%\noindent {\bf\sc Algorithm (A1).} 
%\\[3pt]
%Let us fix the level $L$.\\
%{\bf\rm  Step 1:} Simulate a non-negative random variable $T$ with p.d.f. $f_T$. \\[3pt]
%{\bf\rm Step 2:} Simulate a Poisson random variable $N$ of intensity $\kappa T$.\\[3pt]
%{\bf\rm Step 3:} Simulate two independent sequences of $N$ independent standard uniformly distributed random variables $U:=(U_1,\ldots,U_N)$ and $V:=(V_1,\ldots,V_N)$. We arrange the first sequence in increasing order and denote it $\overline{U}$.We also simulate $N$ independent Gaussian vectors $G_1,\ldots,G_N$ of dimension $3$, whose covariance matrix is the $3\times 3$ identity, and independent of the variables $U$, $V$ and $T$.\\[3pt]
%{\bf\rm Step 4:} For $i=1$ to $N$, we define step by step
%\[
%\beta_i=\frac{1-\overline{U}_i}{1-\overline{U}_{i-1}}\ \beta_{i-1}+\sqrt{\frac{(1-\overline{U}_{i})(\overline{U}_i-\overline{U}_{i-1})}{1-\overline{U}_{i-1}}}\,G_i 
%\]
%where $\overline{U}_0=0$, $\beta_0=(0,0,0)$.\\[3pt]
%{\bf\rm Step 5:} If 
%\[
%V_i\ge \frac{1}{\kappa}\, \gamma(L-\Vert \overline{U}_i(L-x)(1,0,0)+\sqrt{T}\beta_i\Vert )
%\]
%for all $1\le i \le N$ then set $Y=T$ otherwise go to Step 1.\\[5pt]
%{\bf\rm Outcome:} the random variable $Y$.}
%\end{framed}
\begin{framed}\emph{
\noindent {\bf\sc Algorithm (A1).} 
\\[3pt]
Let us fix the level $L>x$.\\
{\bf\rm  Step 1:} Simulate a non-negative random variable $T$ with p.d.f. $f_T$.\\[5pt]
Initialization: $\beta=(0,0,0)$, $\mathcal{W}=0$, $\mathcal{E}_0=0$ and $\mathcal{E}_1$ an exponentially distributed random variable of average $1/\kappa$.\\[3pt]
{\bf\rm Step 2:} While $\mathcal{E}_1\le T$ and $\mathcal{W}=0$,\\[3pt]
{\bf\rm Step 2.1:} Simulate three independent random variables: a $3$-dimensional gaussian random variable $G$, an exponentially distributed variable $e$ with average $1/\kappa$ and a standard uniform random variable $V$.\\[3pt]
{\bf\rm Step 2.2:} Compute 
\[
\beta\leftarrow\frac{T-\mathcal{E}_1}{T-\mathcal{E}_0}\ \beta+\sqrt{\frac{(T-\mathcal{E}_1)(\mathcal{E}_1-\mathcal{E}_0)}{T-\mathcal{E}_0}}\,G. 
\]
Change the value of $\mathcal{W}$, $\mathcal{E}_0$ and  $\mathcal{E}_1$  in the following way: 
\[
\mbox{if}\quad \kappa V\le  \gamma\Big(L-\Vert \mathcal{E}_1(L-x)(1,0,0)/T+\beta\Vert \Big)\ \mbox{then}\ \mathcal{W}=1\ \mbox{else}\ \mathcal{W}=0,
\] 
%\[
%\mathcal{W}\leftarrow 1_{\{ \kappa V\le  \gamma(L-\Vert \mathcal{E}_1(L-x)(1,0,0)/T+\beta\Vert )\}},
%\] 
\centerline{$\mathcal{E}_0\leftarrow \mathcal{E}_1$ and $\mathcal{E}_1\leftarrow\mathcal{E}_1+e$.}\\[5pt]
{\bf\rm Step 3:} If $\mathcal{W}=0$ then set $Y=T$ otherwise go to Step 1.\\[5pt]
{\bf\rm Outcome:} the r.v. $Y$.}
\end{framed}

In order to present the second modified algorithm, we introduce different notations.  Let us consider a finite set $\mathcal{S}\subset\mathbb{R}\times\mathbb{R}^3$ with the property: if both $(t,x)$ and $(t,y)$ belong to $\mathcal{S}$ then $x=y$. Let us moreover define $\underline{\mathcal{S}}=\inf\{t:\ (t,x)\in\mathcal{S}\}$ and $\overline{\mathcal{S}}=\sup\{t:\ (t,x)\in\mathcal{S}\}$. For any $t\in[\underline{\mathcal{S}},\overline{\mathcal{S}}]$ we introduce
$\overleftarrow{t}:=(\overleftarrow{t_1},\overleftarrow{t_2})$ corresponding to the couple $(s,x)\in\mathcal{S}$ such that $\overleftarrow{t_1}=\sup\{s: (s,x)\in\mathcal{S},\ s\le t\}$. Similarly we define   $\overrightarrow{t}:=(\overrightarrow{t_1},\overrightarrow{t_2})$ the couple $(s,x)\in\mathcal{S}$ such that $\overrightarrow{t_1}=\inf\{s: (s,x)\in\mathcal{S},\ s> t\}$.
\begin{framed}\emph{
\noindent {\bf\sc Algorithm (A2).} 
\\[3pt]
Let us fix the level $L>x$.\\
{\bf\rm  Step 1:} Initialization: $\mathcal{E}=0$ and $\mathcal{W}=0$. Simulate a non-negative random variable $T$ with p.d.f. $f_T$. Set $\mathcal{S}=\{(0,(0,0,0)), (T,(0,0,0))\}$.\\[3pt]
{\bf\rm Step 2:} While $\mathcal{E}\le \kappa$ and $\mathcal{W}=0$,\\[3pt]
{\bf\rm Step 2.1:} Simulate two independent random variables: an exponentially distributed variable $e$ with average $1/T$ and a uniformly distributed variable $U$ on the interval $[0,T]$. We change the value of $\mathcal{E}$: $\mathcal{E}\leftarrow\mathcal{E}+e$. \\[3pt]
{\bf\rm Step 2.2:} If $\mathcal{E}\le \kappa$ then  we simulate a $3$-dimensional gaussian random variable $G$ with covariance ${\rm Id}_3$ and we define:
\[
\beta=\overleftarrow{U_2}+\frac{\overrightarrow{U_2}-\overleftarrow{U_2}}{\overrightarrow{U_1}-\overleftarrow{U_1}}\, (U-\overleftarrow{U_1})+\sqrt{\frac{(\overrightarrow{U_1}-U)(U-\overleftarrow{U_1})}{\overrightarrow{U_1}-\overleftarrow{U_1}}}\, G.
\]
We increase the set $\mathcal{S}$ in the following way: $\mathcal{S}\leftarrow\mathcal{S}\cup\{(U,\beta)\}$ and achieve the test:
\[
\mbox{if} \ \mathcal{E}\le\gamma\Big(L-\Vert U(L-x)(1,0,0)/T+%\sqrt{T}
\beta\Vert \Big)\ \mbox{then}\ \mathcal{W}=1\ \mbox{else}\ \mathcal{W}=0.
\]
%If $\mathcal{E}>\kappa$ then we go directly to step 3.\\[3pt]
{\bf\rm Step 3:} If $\mathcal{W}=0$ then $Y=T$ otherwise return to Step 1.\\[5pt]
{\bf\rm Outcome:} $Y$.}
\end{framed}
\begin{prop}\label{prop:A1:conv}
Under Assumption \ref{assu-1}--\ref{assu-3}, if $T/(L-x)^2$ has the same distribution as $1/G^2$ where $G$ is a standard Gaussian r.v., then $Y$, the outcome of Algorithm (A1) (respectively Algorithm (A2)) and $\tau_L$, defined by \eqref{eq:def:tau} \& \eqref{eq:simple-sde}, are identically distributed.
\end{prop}
\begin{proof}
Algorithms (A1) and (A2) are just two different modifications of Algorithm (A0).\\
\emph{Step 1.} Let us first present a quite different expression for $\eta(t)$ defined by \eqref{prop:eq:2} (these arguments were already developed in \cite{ichiba2011efficient}). Let us define 
\begin{equation}\label{eq:defZ}
\mathcal{Z}(y):=\mathbb{E}\Big[ F(R_u,\, 0\le u \le t) \Big | R_t=y \Big],\quad y\ge 0,
\end{equation}
for any non-negative functional $F$. Since the Bessel process $(R_t,\, t\ge 0)$ has the same distribution  as the norm of the 3-dimensional Brownian motion $(B_t,\ t\ge 0)$, we obtain:
\[
\mathcal{Z}(y):=\mathbb{E}\Big[ F(\Vert B_u\Vert,\, 0\le u \le t) \Big | \Vert B_t\Vert =y \Big]=\mathbb{E}\Big[G(B_t)\Big\vert \Vert B_t\Vert =y\Big],
\]
where 
\[
G(x):=\mathbb{E}\Big[ F(\Vert B_u\Vert,\, 0\le u \le t) \Big | B_t=x\Big],\quad x\in\mathbb{R}^3.
\]
The Brownian path can be decomposed as follows:
\[
B_u=\frac{u}{t}\, B_t+\sqrt{t}\beta_{u/t}, \quad 0\le u \le t.
\]
Here $(\beta_s,0\le s\le 1)$ stands for a standard $3$-dimensional Brownian bridge with $\beta_0=\beta_1=0$. We obtain therefore a new expression for the functional $G$:
\[
G(x)=\mathbb{E}\Big[ F(\Vert \frac{u}{t}\, x+\sqrt{t}\beta_{u/t}\Vert,\, 0\le u \le t) \Big]
\]
Since the Brownian bridge is rotationally invariant, we observe 
\(
G(x)=G(\Vert x \Vert e_1)
\) 
where $e_1=(1,0,0)$. We deduce
\begin{equation}\label{eq:defZ1}
\mathcal{Z}(y)=\mathbb{E}\Big[ F(\Vert \frac{u}{t}\ y\ e_1+\sqrt{t}\beta_{u/t}\Vert,\, 0\le u \le t) \Big],
\end{equation}
and in particular:
\[
\eta(t)=\mathbb{E}\Big[ \exp-\int_0^t\gamma \Big(L-\Vert \frac{u}{t}\, (L-x) e_1+\sqrt{t}\beta_{u/t}\Vert\Big)\, \dint u \Big].
\]
\emph{Step 2.} Using the new expression of the weight $\eta(t)$, we propose the same kind of rejection algorithm as (A1). It suffices therefore to consider the domain:
\begin{equation}
\label{Dtilde}
D_{\beta, T}=\Big\{ (t,y)\in[0,T]\times\mathbb{R}_+:\ y\le \gamma\Big(L-\Vert \frac{t}{T}\, (L-x) e_1+\sqrt{T}\beta_{t/T}\Vert\Big) \Big\}
\end{equation}
and the associated event $N(D_{\beta, T})=0$ where $N$ is a Poisson point process on the state space $[0,T]\times[0,\kappa]$. As usual, there exist three methods of simulation for the Poisson process. 
\begin{enumerate}
\item One method consists in first sampling the number of points in the rectangle using the Poisson distribution and secondly in choosing their position uniformly on the rectangle. An algorithm based on such a procedure would be time consuming and requires a huge memory space. It is not reasonable in practice.
\item The second method consists in the simulation of the point process in the infinite domain $[0,T]\times \mathbb{R}_+$: we  simulate a sequence of independent couples $(e_k,U_k)_{k\ge 1}$, $e_k$ and $U_k$ being independent random variables. $e_k$ is exponentially distributed with average $1/T$ and $U_k$ is uniformly distributed on the interval $[0,T]$ for any $k\ge 1$. Let us define:
\[
\eta:=\inf\Big\{n\ge 1: \sum_{k=1}^n e_k>\kappa\Big\}.
\]
The point process therefore corresponds to the couples $(e_1,U_1),\ldots, (e_1+e_2+\ldots+e_{\eta-1},U_{\eta-1})$ belonging to the domain $[0,T]\times[0,\kappa]$. Let us note that $\eta=1$ corresponds to the case where no point of the Poisson process belongs to the domain. Algorithm (A2) is based on this simulation.
\item A third method consists in using similar arguments as those just described but we interchange the space and the time variables. That means that a cumulated sum of exponentially distributed variables with average $1/\kappa$ until it overcomes the level $T$ are used to describe the time variable and uniform variables on the interval $[0,\kappa]$ for the space variable: we obtain Algorithm (A1). 
\end{enumerate}
Once the Poisson process is realized on the domain $[0,T]\times[0,\kappa]$, it suffices to test each point in order to observe if it belongs or not to the domain $D_{\beta, T}$ defined by \eqref{Dtilde}. For this test we use the classical Brownian bridge simulation as described in the fourth step of Algorithm (A1) or Step 2.2 in Algorithm (A2).
\end{proof}
\section{Efficiency of the algorithms}\label{sec:efficiency}

Let us denote by $\mathcal{I}$ the number of iterations observed in order to simulate the hitting time $\tau_L$ defined by \eqref{eq:def:tau} \& \eqref{eq:simple-sde} and $\mathcal{N}_1,\ldots,\mathcal{N}_{\mathcal{I}}$ the numbers of random points (Poisson process) used for each iteration. Of course the efficiency is directly linked to 
\begin{equation}\label{eq:def:nsigm}
\mathcal{N}_\Sigma=\mathcal{I}+\mathcal{N}_1+\ldots+\mathcal{N}_{\mathcal{I}},
\end{equation}
the total number of random variables. It is therefore quite challenging to modify the algorithms in order to reduce it. Focusing on the efficiency, we first describe the number $\mathcal{I}$, secondly we propose some modification in order to reduce it and finally present upper-bound for the total number of random variables $\mathcal{N}_\Sigma$.

\subsection{Number of iterations}
\label{sec:number-ite}
Let us first note that the distribution of the number of iterations $\mathcal{I}$ (number of random variates $T$ produced in order to obtain the outcome $Y$) does not depend on the choice of the algorithm (A0), (A1) or (A2).  The variable $\mathcal{I}$ is exponentially large as soon as $L-x$ becomes large:
\begin{prop}\label{prop:number-of-I}
Under Assumptions \ref{assu-1}--\ref{assu-3},  the following upper-bound holds:
\begin{equation}
\label{prop:upper}
\mathbb{E}[\mathcal{I}]=\exp\{\beta(L)-\beta(x)\}\le \exp((L-x)\sqrt{2\kappa}).
\end{equation}
\end{prop}
\begin{proof}
Let us just note that $\mathcal{I}$ is geometrically distributed with
\[
\mathbb{P}(\mathcal{I}=1)=\mathbb{P}(N(D_{R,T})=0)=\mathbb{E}_x[\eta(\tau_L)].
\]
Since $\tau_L<\infty$ a.s., \eqref{prop:eq:1} leads to 
\(
\mathbb{E}[\mathcal{I}]=\mathbb{E}[\eta(\tau_L)]^{-1}=\exp\{\beta(L)-\beta(x)\}.
\)
Moreover, due to the condition $\gamma\le \kappa$, we obtain
\[
\mathbb{E}_x[\eta(\tau_L)]\ge \mathbb{E}_x[\exp -\kappa\tau_L]=\mathbb{E}\Big[\exp-\frac{\kappa (L-x)^2}{G^2}\Big]=\exp\Big\{-(L-x)\sqrt{2\kappa}\Big\}, % see for instance Borodin p.198
\]
where $G$ is a standard gaussian r.v.
\end{proof}
In the simulation framework, such an exponential large number of iterations seems crippling in practice for large values of $L$ and stimulates one to find interesting modifications of Algorithms (A1)--(A2).
\subsubsection*{Space splitting}
In order to reduce the number of iterations, we could replace the simulation of $\tau_L$ for a diffusion starting in $x$ (denoted here by $\tau(x\to L)$) by the simulation of $k$ independent first-passage times $\tau(x+(i-1)(L-x)/k\to x+i(L-x)/k)$ for $i=1,\ldots,k$. It suffices then to take the sum of all these quantities. The global number of iterations for this new algorithm, called Algorithm ${\rm (A1)_{split}}$ or ${\rm (A2)_{split}}$, becomes:
\begin{align*}
\mathbb{E}[\mathcal{I}_{\rm split}]&=\sum_{i=1}^k\exp\Big\{\beta\Big(x+i(L-x)/k\Big)-\beta\Big(x+(i-1)(L-x)/k\Big)\Big\}\\
&\le k \exp\Big(\frac{(L-x)}{k}\sqrt{2\kappa}\Big)\le (\lfloor (L-x)\sqrt{2\kappa} \rfloor+1)e.
\end{align*}
The previous inequality is related to the particular choice $k=(\lfloor (L-x)\sqrt{2\kappa} \rfloor+1)$. Such a procedure permits to replace an exponentially large number of iterations into a linear one, as $L$ becomes large. 
\mathversion{bold}
\subsubsection*{Shifting the function $\gamma$}
\mathversion{normal}

An other way to reduce $\mathcal{I}$ is to decrease the parameter $\kappa$.  Such a trick is possible in the following situation.
\begin{assumption}\label{assu-4} There exist two constants $\kappa>\gamma_0>0$ such that 
\[
\kappa\ge \gamma(x)\ge \gamma_0,\quad \forall x\in]-\infty,L].
\]
\end{assumption}
The modification is based on the following observation: \eqref{prop:eq:1} can be replaced by
\begin{equation}\label{prop:eq:1bis}
\mathbb{E}_\mathbb{P}[\psi(\tau_L)1_{\{ \tau_L<\infty \}}]=\mathbb{E}_\mathbb{Q}[\psi(\tau_L)\eta_0(\tau_L)e^{-\gamma_0 \tau_L}]\exp \Big\{\beta(L)-\beta(x)\Big\},
\end{equation}
$\eta_0(t)$ being defined in a similar way as $\eta(t)$ in \eqref{prop:eq:2} just by replacing $\gamma(t)$ by the shifted function $\gamma(t)-\gamma_0$. In other words, the upper-bound parameter $\kappa$ can be replaced by the shifted value $\kappa-\gamma_0$ and consequently the number of iterations is reduced. 
\begin{prop} 
\label{prop:inv:gauss}
Let Algorithm $(A1)_{\rm shift}$, respectively $(A2)_{\rm shift}$, defined in a similar way as Algorithm $(A1)$, resp. $(A2)$, just by replacing  the function $\gamma(\cdot)$  by the shifted function $\gamma(\cdot)-\gamma_0$. Under Assumption \ref{assu-1}, \ref{assu-2} and \ref{assu-4}, if $T$ is an inverse gaussian random variable ${\rm IG}(\frac{L-x}{\sqrt{2\gamma_0}},(L-x)^2)$ then the outcome of the modified algorithm $Y$ and $\tau_L$ defined by \eqref{eq:def:tau} \& \eqref{eq:simple-sde} are identically distributed. Moreover this modification reduces the number of iterations:
\begin{equation}
\label{eq:invg}
\mathbb{E}[\mathcal{I}_{\rm shift}]=\mathbb{E}[\mathcal{I}] e^{-(L-x)\sqrt{2\gamma_0}}\quad\mbox{for}\quad x<L.
\end{equation}
\end{prop}
\begin{proof}
Let us consider a non-negative bounded function $f$ and introduce 
\[
E_f:=\mathbb{E}_x\Big[ f(\tau_L)\exp-\gamma_0\tau_L \Big],
\]
where $\tau_L$ is the Brownian first-passage time, that is $\tau_L$ and $(L-x)^2/G^2$ are identically distributed (here $G$ stands for a standard gaussian r.v.). Hence
\begin{align*}
E_f&=\int_0^\infty f(t)\frac{L-x}{\sqrt{2\pi t^3}}\,\exp-\Big\{\frac{(L-x)^2}{2t}+\gamma_0 t\Big\} \, \dint t\\
&=e^{-\sqrt{2\gamma_0}(L-x)}\int_0^\infty f(t)\frac{L-x}{\sqrt{2\pi t^3}}\,\exp-\Big\{\frac{(L-x-\sqrt{2\gamma_0}t)^2}{2t}\Big\}\, \dint t
\end{align*}
We deduce that $E_f/E_1=\mathbb{E}[f(T)]$ where $T\sim {\rm IG}(\frac{L-x}{\sqrt{2\gamma_0}},(L-x)^2)$. Let us denote by $f_T$ the \emph{p.d.f.} of the variable $T$. By Theorem \ref{thm:dens}, we get
\begin{align*}
\mathbb{E}[\psi(Y)]&=\left(\int_0^\infty \psi(t)\eta_0(t) f_T(t)\,\dint t\right)\left(\int_0^\infty \eta_0(t) f_T(t)\,\dint t\right)^{-1}\\
&=\frac{E_{\psi\eta_0}}{E_1}\Big(\frac{E_{\eta_0}}{E_1}\Big)^{-1}=\frac{E_{\psi\eta_0}}{\mathbb{E}_x[\eta(\tau_L)]}.
\end{align*}
Let us recall that $\mathbb{E}_x[\eta(\tau_L)]=\exp-\{ \beta(L)-\beta(x) \}$. Then \eqref{prop:eq:1bis} permits to point out the required distribution.
Let us just notice that using the inverse gaussian variable increases the conditional probability of acceptance:
\[
\mathbb{P}(\mathcal{I}_{\rm mod}=1|T)=\mathbb{P}(N(D_{R,T})=0|T)=\eta_0(T)>\eta(T).
\]
Consequently the probability of acceptance/rejection at each step changes: $\mathcal{I}_{\rm mod}$ is geometrically distributed and
\begin{align*}
&\mathbb{P}(\mathcal{I}_{\rm shift}=1)=\mathbb{P}(N(D_{R,T})=0)=\mathbb{E}[\eta_0(T)]=\frac{\mathbb{E}[\eta_0(\tau_L)e^{-\gamma_0\tau_L}]}{\mathbb{E}[e^{-\gamma_0\tau_L}]}\\
&\quad=\mathbb{E}[\eta(\tau_L)]\exp\Big\{(L-x)\sqrt{2\gamma_0}\Big\}=\exp-\Big\{ \beta(L)-\beta(x)-(L-x)\sqrt{2\gamma_0} \Big\}.
\end{align*}
\end{proof}
\begin{rem} Using a many-to-one transformation, Michael, Schucany and Haas introduced a simple generator of inverse gaussian distributions (see, for instance, \cite{Devroye-1986} p.149). We can apply the following procedure to simulate $T\sim {\rm IG}(\mu,\lambda)$: let $N$ be a standard gaussian r.v. and $U$ uniformly distributed on $[0,1]$, independent of $N$. 
\begin{enumerate}
\item Set $X=\mu+\frac{\mu^2 N^2}{2\lambda}-\frac{\mu}{2\lambda}\,\sqrt{4\mu \lambda N^2+\mu ^2 N^4}$.
\item If $U\le \frac{\mu }{\mu +X}$ then $T=X$ else $T=\frac{\mu ^2}{X}$.
\end{enumerate}
\end{rem}

\subsection{Efficiency of Algorithm (A1)}
In this section, we are looking for an upper-bound of the number of random variates $\mathcal{N}_\Sigma$ used in order to simulate one variate $\tau_L$. First we shall focus on the first iteration in Algorithm (A1) and secondly adapt the procedure to the whole number of variates.
%Let $\mathcal{N}_1$ the number of simulated points of the point process during the first iteration of the Algorithm (A1).
\begin{prop}\label{prop:iter} Let $\gamma$ satisfy Assumptions \ref{assu-1}--\ref{assu-3}.  We assume moreover the existence of  $C_\gamma>0$ and $r<1$ such that 
\begin{equation}
\label{eq:condprop}
\inf_{y\le z\le L}\gamma(z)\ge C_\gamma |y|^{-r},\quad \mbox{for all }\quad y\le -1. 
\end{equation}
%Let us consider Algorithm (A1) with $T$ corresponding to the first passage-time to the level $L$ for the Brownian motion starting in $x<L$.
Then there exist two constants $M_{\gamma,1}>0$ and $M_{\gamma,2}>0$ such that the number of random points in the first iteration of Algorithm (A1),  satisfies
\[
\mathbb{E}[\mathcal{N}_1]\le M_{\gamma,1}+\kappa M_{\gamma,2}(x^2+(L-x)^{(1+r)/2}),\quad \mbox{for}\ x<L.
\]
\end{prop}
\begin{proof} Let us denote by $\mathbb{P}_{c}$ (resp. $\mathbb{E}_{c}$) the conditional probability measure (resp. expectation) given both $\tau_L=T$ and the Bessel bridge trajectory $(R_t,0\le t\le T)$.  We observe that
\[
\mathbb{P}_{c}(\mathcal{N}_1=1)=\mathbb{P}_{c}(\epsilon_1>T)+\mathbb{P}_{c}(U_1<\gamma(L-R_{\epsilon_1}),\epsilon_1\le T),
\]
where $(U_k)_k$ is a sequence of uniformly distributed r.v. on $[0,\kappa]$ and $(\epsilon_k)_k$ is a sequence of exponentially distributed r.v. with average $\frac{1}{\kappa}$, both sequences being independent.  We denote $\mathcal{E}_k=\epsilon_1+\ldots+\epsilon_k$. Hence
\[
\mathbb{P}_{c}(\mathcal{N}_1=1)=e^{-\kappa T}+  \mathbb{E}_{c}\Big[\frac{\gamma}{\kappa}(L-R_{\epsilon_1})1_{\{ \epsilon_1\le T \}}\Big]=e^{-\kappa T}+\int_0^T\gamma(L-R_t)e^{-\kappa t}\,\dint t.
\]
%Moreover, given $T$ and the Bessel path,
%\[
%\mathbb{P}_c(\mathcal{N}_1=1,A^c)=\mathbb{P}_c(U_1<\gamma(L-R_{\epsilon_1}),\epsilon_1\le T)=\int_0^T\gamma(L-R_t)e^{-\kappa t}\,dt.
%\]
%Since $\mathbb{P}_c(A^c)=1-\eta(R,T):=1-\exp-\int_0^T\gamma(L-R_t)\,dt$, we obtain
%\begin{eqnarray*}
%\left \{\begin{array}{l}
%\mathbb{P}_c(\mathcal{N}_1=1|A^c)=\left(\int_0^T\gamma(L-R_t)e^{-\kappa t}\,dt\right)\left( 1-\eta(R,T) \right)^{-1},\\[8pt]
%\mathbb{P}_c(\mathcal{N}_1=1|A)=e^{-\kappa T}/\eta(R,T).
%\end{array}\right.\end{eqnarray*}
By similar computations, for $k\ge 1$,
%\begin{align*}
%\mathbb{P}_c(\mathcal{N}_1>1)&=\mathbb{P}_c(U_1\ge \gamma(L-R_{\epsilon_1}),\epsilon_1\le T)=\frac{1}{\kappa}\mathbb{E}_c[(\kappa-\gamma(L-R_{\epsilon_1}))1_{\{ \epsilon_1\le T \}}]
%\end{align*}
\begin{align*}
\mathbb{P}_c(\mathcal{N}_1>k)&=\mathbb{P}_c(U_1\ge \gamma(L-R_{\epsilon_1}),\ldots, U_k\ge \gamma(L-R_{\mathcal{E}_k}),\mathcal{E}_k\le T)\\
&=\frac{1}{\kappa^k}\ \mathbb{E}_c\left[\prod_{i=1}^k(\kappa-\gamma(L-R_{\mathcal{E}_i}))1_{\{ \mathcal{E}_k\le T \}}\right]\\
&=\frac{1}{\kappa^k}\ \mathbb{E}_c\left[\prod_{i=1}^{k-1}(\kappa-\gamma(L-R_{\mathcal{E}_k U^{(i)}}))(\kappa-\gamma(L-R_{\mathcal{E}_k}))1_{\{ \mathcal{E}_k \le T \}}\right],
\end{align*}
where $(U^{(1)},\ldots,U^{(k)})$ stands for a set of reordered uniform random variables and $\mathcal{E}_k$ is gamma distributed $\gamma(k,\kappa)$.
Hence 
\begin{align*}
\mathbb{P}_c(\mathcal{N}_1>k)&=\mathbb{E}\left[\Big(1-\int_0^1 \frac{\gamma}{\kappa}(L-R_{\xi s})\,\dint s\Big)^{k-1}\Big(1-\frac{\gamma}{\kappa}(L-R_{\xi})\Big)1_{\{ \xi \le T \}}\right]\\
&=\int_0^T \Big(1-\int_0^1 \frac{\gamma}{\kappa}(L-R_{u s})\,\dint s\Big)^{k-1}\Big(1-\frac{\gamma}{\kappa}(L-R_{u})\Big)\frac{u^{k-1}\kappa^k}{(k-1)!}\, e^{-\kappa u}\, \dint u
\end{align*}
Consequently, using a change of variable, we obtain
\begin{align}\label{eq:1}
\mathbb{E}_c[\mathcal{N}_1]&=\sum_{k=0}^\infty \mathbb{P}_c(\mathcal{N}_1>k)\nonumber \\
%&=1+\sum_{k=1}^\infty \int_0^T \frac{ (u\kappa)^{k-1}}{(k-1)!}\Big(1-%\int_0^1 \frac{\gamma}{\kappa}(L-R_{u s})\,ds\Big)^{k-1}\Big(\kappa-%%\gamma(L-R_{u})\Big)e^{-\kappa u}\, du\nonumber \\
&=1+\int_0^T\Big(\kappa-\gamma(L-R_{u})\Big)\exp\Big\{-u\int_0^1\gamma(L-R_{us})\,ds\Big\}\, \dint u\nonumber \\
&=1+\int_0^T\Big(\kappa-\gamma(L-R_{u})\Big)\exp\Big\{-\int_0^u\gamma(L-R_{w})\, \dint w\Big\}\, \dint u.
%&=1+\int_0^T\Big(\kappa-\gamma(L-R_{u})\Big)\exp\Big\{\kappa u-\int_0^u\gamma(L-R_{w})\,dw-\kappa u\Big\}\, du.
\end{align}
By integration by parts, \eqref{eq:1} becomes $\mathbb{E}_c[\mathcal{N}_1]=H_T+\kappa I_T$ where
\begin{equation}\label{eq::1}
H_T:=e^{-\int_0^T\gamma(L-R_{w})\,\dint w},\quad I_T:=\int_0^T \exp\Big\{-\int_0^u\gamma(L-R_{w})\,\dint w\Big\}\, \dint u.
\end{equation}
Let us first study the expression $I_T$. Using the scaling property of the Bessel bridge, we have
\begin{align*}
I_T&=T\int_0^1\exp\Big\{-T\int_0^u\gamma(L-\sqrt{T}\hat{R}_{w}^T)\,\dint w\Big\}\, \dint u,
\end{align*}
where $(\hat{R}_w^T,\, 0\le w\le 1)$ is a $3$-dimensional Bessel bridge with condition $\hat{R}_1^T=(L-x)/\sqrt{T}$. Let us observe that 
\begin{equation}\label{eq:stochord}
\hat{R}_w\le_{st} \frac{L-x}{\sqrt{T}}+\bar{R}_w,\quad 0\le w\le 1.
\end{equation}
Here $(\bar{R})$ is a standard $3$-dimensional Bessel bridge. This stochastic ordering result  can be proven using the following arguments: first the Bessel bridge $\hat{R}^T$ has the same distribution as the Euclidian norm of $\frac{(L-x)w}{\sqrt{T}}\,e_1+\beta_w$ where $\beta_w$ is a standard $3$-dimensional Brownian bridge, see \eqref{eq:defZ} and \eqref{eq:defZ1}. Secondly the triangle inequality leads to 
\[
\Big\Vert \frac{(L-x)w}{\sqrt{T}}\,e_1+\beta_w\Big\Vert \le \Big\Vert \frac{(L-x)w}{\sqrt{T}}\,e_1\Big\Vert +\Vert \beta_w\Vert\le \frac{L-x}{\sqrt{T}}+\Vert\beta_w\Vert.
\]
In order to prove the stochastic ordering \eqref{eq:stochord} it suffices then to note that $\Vert \beta_w\Vert$ and $\bar{R}_w$ are identically distributed.

Let us now introduce some parameter $\alpha\in]0,1[$ and define 
\begin{equation}\label{eq:defgammaT}
\gamma_T=\inf_{y\in[x-T^{1/2+\alpha},L]}\gamma(y).
\end{equation}
Then
\begin{align*}
I_T&\le T\int_0^1 e^{-T\gamma_T u}\,\dint u\ \mathbb{P}\Big( \sup_{u\in[0,1]}\sqrt{T}\hat{R}^T_u\le T^{1/2+\alpha}+(L-x) \Big) \\
&+T\ \mathbb{P}\Big( \sup_{u\in[0,1]}\sqrt{T}\hat{R}^T_u> T^{1/2+\alpha}+(L-x) \Big). 
\end{align*}
By \eqref{eq:stochord} we obtain:
\begin{align}\label{eq:sum}
I_T&\le \frac{1}{\gamma_T}1_{\{ T\ge (x+1)^{2/(1+2\alpha)} \}}+T1_{\{ T< (x+1)^{2/(1+2\alpha)} \}} +T\ \mathbb{P}\Big( \sup_{u\in[0,1]}\bar{R}_u> T^{\alpha} \Big) \nonumber\\
&\le \frac{1}{\gamma_T}1_{\{ T\ge (x+1)^{2/(1+2\alpha)} \}}+
(x+1)^{2}
+T\ \mathbb{P}\Big( \sup_{u\in[0,1]}\bar{R}_u> T^{\alpha} \Big) .
\end{align}
Using the agreement formula (see, for instance, Theorem 2 in \cite{pitman-yor}), we obtain
\[
 \mathbb{P}\Big( \sup_{u\in[0,1]}\bar{R}_u> T^{\alpha} \Big) =C_3\mathbb{E}\Big[\sqrt{\bar{\tau}}1_{\{ \bar{\tau} < T^{-2\alpha} \}}\Big].
\]
Here $C_3=\sqrt{2}/\Gamma(3/2)$ and $\bar{\tau}=\tau+\hat{\tau}$ where $\tau$ is the first hitting time of the level $1$ for a $3$-dimensional Bessel process and $\hat{\tau}$ an independent copy of $\tau$. Combining classical inequalities and the explicit expression of the Laplace transform corresponding to the distribution of $\tau$ (see \cite{kent}) leads to
\begin{align*}
& \mathbb{P}\Big( \sup_{u\in[0,1]}\bar{R}_u> T^{\alpha} \Big)\le C_3 T^{-\alpha}\, \mathbb{P}\Big(\exp-\lambda \bar{\tau} > \exp -\lambda T^{-2\alpha}\Big)\\
& \le  C_3T^{-\alpha}\, e^{\lambda T^{-2\alpha}}\mathbb{E}[e^{-\lambda \bar{\tau}}]=C_3T^{-\alpha}\, e^{\lambda T^{-2\alpha}}\frac{(2\lambda)^{1/2}}{C_3^2 I_{1/2}^2(\sqrt{2\lambda})},
\end{align*}
for any $\lambda>0$.  $I_\nu$ stands for the Bessel function of the first kind. In particular $I_{1/2}(x)=\sqrt{\frac{2}{\pi x}}\,\sinh x$. The particular choice $\lambda=T^{2\alpha}/2$ leads to
\begin{align*}
 \mathbb{P}\Big( \sup_{u\in[0,1]}\bar{R}_u> T^{\alpha} \Big)&\le
 \frac{\sqrt{e\pi}}{2\sqrt{2}} \,\frac{\pi T^{\alpha}}{2\sinh^2(T^{\alpha})}.
\end{align*}
There exists a constant $C_{\alpha}>0$ such that
\[
 T\ \mathbb{P}\Big( \sup_{u\in[0,1]}\bar{R}_u> T^{\alpha} \Big)\le\frac{\sqrt{e}}{2}\Big(\frac{\pi}{2}\Big)^{3/2}\sup_{y\ge 0}\frac{y^{1+\alpha}}{\sinh^2(y^{\alpha})}\le C_{\alpha}.
\]
Let us go back to the upper-bound \eqref{eq:sum}. By definition, $\gamma_T\ge C_\gamma |x-T^{1/2+\alpha}|^{-r}$  as soon as $T\ge (x+1) ^{2/(1+2\alpha)}$ and therefore 
\[
I_T\le \frac{1}{C_\gamma}|x-T^{1/2+\alpha}|^{r}1_{\{ T\ge (x+1)^{2/(1+2\alpha)} \}}+(x+1)^2+C_\alpha.
\]
Let us recall that $T$ is the Brownian first-passage time through the level $L$ starting in $x$: $T$ and $\frac{(L-x)^2}{G^2}$ are identically distributed,  $G$ being a standard gaussian variable. Taking now the expectation with respect to $T$, we get
\begin{align*}
\mathbb{E}[\mathcal{N}_1]&\le 1+\kappa \mathbb{E}[I_T]\\
&\le 1+\frac{\kappa}{C_\gamma}\mathbb{E}[|x-T^{1/2+\alpha}|^{r}1_{\{ T\ge (x+1)^{2/(1+2\alpha)} \}}]+\kappa(x+1)^2+\kappa C_\alpha\\
&\le 1+\frac{\kappa 2^r}{C_\gamma}\, \Big( |x|^r+\mathbb{E}[T^{(1/2+\alpha)r}] \Big)+\kappa(x+1)^2+\kappa C_\alpha\\
&= 1+\frac{\kappa 2^r}{C_\gamma}\, \Big( |x|^r+|L-x|^{(1+2\alpha)r}\mathbb{E}[G^{-(1+2\alpha)r}] \Big)+\kappa(x+1)^2+\kappa C_\alpha.
\end{align*}
Since $r<1$, it is possible to choose $\alpha\in]0,1[$ small enough such that $(1+2\alpha)r<1$ and therefore $\mathbb{E}[G^{-(1+2\alpha)r}] <\infty$. Let us choose for instance $\alpha=\frac{1-r}{4r}$ that means $(1+2\alpha)r=\frac{1+r}{2}$
\end{proof}
\begin{prop}\label{prop:conv:modif} Let us consider Algorithm $(A1)_{\rm shift}$ with input variable $T\sim {\rm IG}(\frac{L-x}{\sqrt{2\gamma_0}},(L-x)^2)$ as in Proposition \ref{prop:inv:gauss}. Then the number of random points associated to the first iteration satisfies:
\[
\mathbb{E}[\mathcal{N}_1]\le 1+\kappa\frac{L-x}{\sqrt{2\gamma_0}}.
\]
\end{prop}
\begin{proof}
The arguments are similar to those developed in the proof of Proposition \ref{prop:iter}. The formula \eqref{eq::1} leads to $\mathbb{E}_c[\mathcal{N}_1]\le 1+\kappa T$. Since $T$ is an inverse gaussian r.v. we obtain the announced upper-bound:
\[
\mathbb{E}[\mathcal{N}_1]\le 1+\kappa\mathbb{E}[T]=1+\kappa\frac{L-x}{\sqrt{2\gamma_0}}<\infty.
\]
\end{proof}
The main description of the algorithm efficiency encompasses both the information about the first iteration and the description of the number of iterations.
\begin{thm}\label{thm:averagenumber} Let us denote by $\mathcal{N}_\Sigma$ the total number of random points used in the simulation of the diffusion first-passage time. 
\begin{enumerate}\item Under condition \eqref{eq:condprop}, the number of points of Algorithm (A1), in the context of Proposition \ref{prop:A1:conv}, satisfies:
\[
\mathbb{E}[\mathcal{N}_\Sigma]\le 
2\Big( M_{\gamma,1}+\kappa M_{\gamma,2}(x^2+(L-x)^{(1+r)/2})\Big)\ e^{\beta(L)-\beta(x)},
\]
 $M_{\gamma,1}$ and $M_{\gamma,2}$ being the constants introduced in Proposition \ref{prop:iter}.
 \item The number of points of the modified algorithm $(A1)_{\rm shift}$, in the context of Proposition \ref{prop:inv:gauss}, satisfies:
 \[
 \mathbb{E}[\mathcal{N}_\Sigma]\le 2\Big(1+\kappa\frac{L-x}{\sqrt{2\gamma_0}}\Big)\ e^{\beta(L)-\beta(x)-(L-x)\sqrt{2\gamma_0} }.
 \]
 \end{enumerate}
\end{thm}
Let us notice in particular that the average number of random points is finite  which actually plays an important role for numerical purposes. 
\begin{proof}
We shall just prove the first item, the second one is based on similar arguments. Let us recall that $\mathcal{I}$ is the number of iterations in Algorithm (A1) and $A$ stands for the event \emph{Acceptance} in the acceptance-rejection method. Therefore
\begin{align}\label{eq:dev}
\mathbb{E}[\mathcal{N}_\Sigma]&=\mathbb{E}[\mathcal{N}_1+\ldots+\mathcal{N}_\mathcal{I}]=\sum_{k=1}^\infty\mathbb{E}[\mathcal{N}_1+\ldots+\mathcal{N}_k|\mathcal{I}=k]\mathbb{P}(\mathcal{I}=k)\nonumber \\
&=\sum_{k=1}^\infty \Big((k-1)\mathbb{E}[\mathcal{N}_1|A^c]+\mathbb{E}[\mathcal{N}_1|A]\Big)\ \mathbb{P}(\mathcal{I}=k).
\end{align}
Let us denote by $p=e^{-(\beta(L)-\beta(x))}$ and observe that $\mathbb{P}(A)=p$. Moreover we get the following upper bound
\[
\mathbb{E}[\mathcal{N}_1|A^c]=\frac{\mathbb{E}[\mathcal{N}_11_{A^c}]}{1-p}\le \frac{\mathbb{E}[\mathcal{N}_1]}{1-p}.
\]
We deduce that \eqref{eq:dev} becomes
\begin{align*}
\mathbb{E}[\mathcal{N}_\Sigma]&\le \mathbb{E}[\mathcal{N}_1]\Big( \sum_{k=1}^\infty \Big\{\frac{k-1}{1-p}+\frac{1}{p}\Big\} (1-p)^{k-1}p \Big)=\frac{2}{p}\,\mathbb{E}[\mathcal{N}_1].
\end{align*}
Proposition \ref{prop:iter} permits to conclude the proof.
\end{proof}

\subsection{Efficiency of Algorithm (A2)}
 In this section, $\mathcal{N}_1$ still represents the number of random points used for the first iteration of the algorithm. 

\begin{thm}\label{thm:iter-2} Let $\gamma$ satisfies Assumptions \ref{assu-1}--\ref{assu-3} and the particular condition  
\eqref{eq:condprop}. We consider Algorithm (A2) with $T$ corresponding to the Brownian first-passage time to the level $L$, with $x<L$.
Then there exist two constants $M_{\gamma,1}>0$ and $M_{\gamma,2}>0$ such that the total number of random points $\mathcal{N}_\sigma$ in Algorithm (A2) satisfies
\[
\mathbb{E}[\mathcal{N}_\Sigma]\le 
2\Big( M_{\gamma,1}+\kappa M_{\gamma,2}(x^2+(L-x)^{(1+r)/2})\Big)\ e^{\beta(L)-\beta(x)}.
\]
\end{thm}
The statement of Theorem \ref{thm:iter-2} is similar to Theorem \ref{thm:averagenumber} but the arguments of proof are slightly different.
\begin{proof}
We recall that $\mathbb{P}_c$ stands for the conditional probability given the first-passage time $T$ and the Bessel bridge $(R_t)$ ending with $R_T=L-x$. We observe that
 \begin{equation}\label{eq:dev1}
 \mathbb{P}_c(\mathcal{N}_1>1)=\mathbb{P}_c(\gamma(L-R_{U_1})\le e_1\le \kappa),
 \end{equation}
 where $U_1$ is uniformly distributed on the interval $[0,T]$ and $e_1$ is an exponentially distributed random variable independent of $U_1$ and with average $1/T$. Hence \eqref{eq:dev1} becomes
\begin{align}
\label{eq:dev2}
&\mathbb{P}_c(\mathcal{N}_1>1)=\int_0^T\int_0^\infty 1_{\{ \gamma(L-R_{u})\le t\le \kappa \}}\ e^{-Tt}\,\dint u\,\dint t\nonumber\\
&\quad \quad=\frac{1}{T}\, \int_0^T \Big(  e^{-T\gamma(L-R_u)}-e^{-T\kappa} \Big)\ \dint u=\mathbb{E}_c[e^{-T\gamma(L-R_{U_1})}]-e^{-T\kappa}.
\end{align} 
Let us generalize these computations.
We introduce $(e_n,\ n\ge 0)$ a sequence of exponentially distributed random variables with average $1/T$, $\mathcal{E}_n:=e_1+\ldots+e_n$ and $(U_n,\ n\ge 0)$ a sequence of uniform distributed independent r.v. on the interval $[0,T]$, all r.v. being independent. Moreover let us denote by $F_c(t)=\gamma(L-R_t)$ for any $t\in[0,T]$ (deterministic function under the conditional distribution $\mathbb{P}_c$) and $G_c(y)=\frac{1}{T}\int_0^T1_{\{ y\ge F_c(t) \}}\,dt$.  For $k>1$, we get
\begin{align}
\label{eq:dev3}
\mathbb{P}_c(\mathcal{N}_1>k)&=\mathbb{P}_c\Big(F_c(U_1)\le \mathcal{E}_1,\ F_c(U_2)\le \mathcal{E}_2,\ldots, F_c(U_k)\le \mathcal{E}_k<\kappa\Big).
\end{align}
 Let us note that, under $\mathbb{P}_c$, we have on one hand $\mathcal{E}_k\sim\gamma(k,\frac{1}{T})$ and on the other hand:
 \[
 \Big(\mathcal{E}_1,\mathcal{E}_2,\ldots,\mathcal{E}_{k-1},\mathcal{E}_k\Big)\sim \Big(V^{(1)}\mathcal{E}_k, V^{(2)}\mathcal{E}_k,\ldots,V^{(k-1)}\mathcal{E}_k,\mathcal{E}_k\Big)
\] 
where $(V^{(1)} ,V^{(2)} ,\dots,V^{(k-1)} )$ are ordered variables corresponding to the sequence of uniform random on $[0,1]$: $(V_1,V_2,\ldots,V_{k-1}).$ 
\begin{align*}
\mathbb{P}_c(\mathcal{N}_1>k)&=\mathbb{P}_c(F_c(U_1)\le V^{(1)}\mathcal{E}_k,\ F_c(U_2)\le V^{(2)}\mathcal{E}_k,\ldots, F_c(U_k)\le \mathcal{E}_k<\kappa)\\
&=\mathbb{E}_c\Big[ G_c(V^{(1)}\mathcal{E}_k)G_c(V^{(2)}\mathcal{E}_k) \ldots G_c(V^{(k-1)}\mathcal{E}_k)G_c(\mathcal{E}_k)1_{\{ \mathcal{E}_k\le \kappa \}} \Big]\\
&=\mathbb{E}_c\Big[ G_c(V_1\mathcal{E}_k)G_c(V_2\mathcal{E}_k) \ldots G_c(V_{k-1}\mathcal{E}_k)G_c(\mathcal{E}_k)1_{\{ \mathcal{E}_k\le \kappa \}} \Big]\\
&=\mathbb{E}_c\Big[ \mathcal{G}_c^{k-1}(\mathcal{E}_k)G_c(\mathcal{E}_k)1_{\{ \mathcal{E}_k\le \kappa \}} \Big].
\end{align*}
Here $\mathcal{G}_c$ stands for $\mathcal{G}_c(y)=\int_0^1 G_c(ty)\,\dint t$.
\begin{align*}
\mathbb{P}_c(\mathcal{N}_1>k)&=\frac{T^k}{\Gamma(k)}\int_0^\kappa x^{k-1} \mathcal{G}_c^{k-1}(x)G_c(x)\ e^{-Tx}\ \dint x.
\end{align*}
We deduce
\begin{align*}
\mathbb{E}_c[\mathcal{N}_1]&=\sum_{k\ge 0}\mathbb{P}_c(\mathcal{N}_1>k)=1+T\sum_{k\ge 0}\int_0^\kappa \frac{(xT\mathcal{G}_c(x))^k}{k!}\ G_c(x) e^{-Tx}\ \dint x\\
&=1+T\int_0^\kappa G_c(x) \exp\Big\{ xT(\mathcal{G}_c(x)-1) \Big\}\ \dint x. 
\end{align*}
Taking the expectation with respect to both $T$ and the Bessel bridge $(R_t)$, we obtain
\begin{equation}
\label{eq:expect=}
\mathbb{E}[\mathcal{N}_1]=1+\int_0^\kappa \mathbb{E}\Big[T\ G_c(x) \exp\Big\{ xT(\mathcal{G}_c(x)-1) \Big\}\Big]\ \dint x.
\end{equation}
Let us now consider some upper-bound. First we note that $0\le G_c(x)\le 1$ for any $x\in\mathbb{R}$. Let us now consider $\mathcal{G}_c(x)$ using a change of variable:
\begin{align*}
\mathcal{G}_c(x)&=\frac{1}{T}\int_0^1\Big(\int_0^T1_{\{ ux\ge \gamma(L-R_t) \}}\,\dint t\Big)\ \dint u\\
&\overset{\Delta}{=}\int_0^1\int_0^1 1_{\{ ux\ge \gamma(L-\sqrt{T}\hat{R}_w^T) \}}\dint w\,\dint u,
\end{align*}
where $(\hat{R}_w^T,\, 0\le w\le 1)$ is a $3$-dimensional Bessel bridge  with condition $\hat{R}_1^T=(L-x)/\sqrt{T}$, as considered in the proof of Proposition \ref{prop:iter}. For $x>0$,
\[
\mathcal{G}_c(x)\le 1-\frac{1}{x}\inf_{y\in[L-\sqrt{T}\Pi_T,L]}\gamma(y),\quad
\mbox{with}\quad \Pi_T=\sup_{w\in[0,1]}\hat{R}_w^T.
\]
By \eqref{eq:expect=} and the previous inequalities, we obtain:
\begin{equation}\label{eq:adapt}
\mathbb{E}[\mathcal{N}_1]\le 1+\int_0^\kappa \mathbb{E}\Big[T \exp-T \inf_{y\in[L-\sqrt{T}\Pi_T,L]} \gamma(y)\Big]\,\dint x.
\end{equation}
Let us just note that the variable $\Pi_T$ depends on $x$.  Using the obvious upper-bound $Te^{-T\gamma_T}\le \int_0^1 Te^{-T\gamma_T u}\,\dint u$ and the definition of $\gamma_T$ in \eqref{eq:defgammaT}, we observe that all arguments developed in the proof of Proposition \ref{prop:iter} and in Theorem \ref{thm:averagenumber}, from \eqref{eq:defgammaT} on, can be adapted to  \eqref{eq:adapt}. 
\end{proof}

\mathversion{bold}
\section{When the diffusion hits the level $L$ before a given fixed time $t_0$}\label{sec:before} \mathversion{normal}
Both Algorithm (A1) and (A2) concern first-passage times for the diffusion $(X_t,\, t\ge 0)$ defined in \eqref{eq:simple-sde} under some particular conditions.  The boundedness of $\gamma:=\frac{b^2+b'}{2}$ on $]-\infty,L]$ is indisputably the most restrictive condition (Assumption \ref{assu-3}). Unfortunately the result presented so far cannot be generalized to the particular situation when $\gamma$ takes some negative values. Nevertheless we can adapt the algorithms by considering $\tau_L|\tau_L\le t_0$ instead of $\tau_L$, defined by \eqref{eq:def:tau}, for any fixed time $t_0$. 
 
 Let us consider a modification of the theoretical Algorithm $(A_0)$. The idea is to focus our attention to random times which are a.s. smaller than some given value $t_0$.  
\begin{assumption}\label{assu-5} There exists two constants $\kappa>0>-m$ such that
\[
-m\le \gamma(x)\le \kappa,\quad \forall x\le L.
\]
\end{assumption} 
 
\begin{framed}\emph{
\noindent {\bf\sc Algorithm (A3).} 
\\[3pt]
Let us fix the level $L>x$, $x$ being the initial value of the diffusion.\\
{\bf\rm  Step 1:} Simulate a non-negative random variable $T$ with p.d.f. $f_T$ of support $[0,t_0]$. \\[3pt]
{\bf\rm Step 2:} Simulate a $3$-dimensional Bessel process $(R_t)$ on the time interval $[0,T]$ with endpoint $R_T=L-x$. We define by $D_{R,T}^m$ the stochastic domain satisfying:
\[
D_{R,T}^m:=\Big\{(t,v)\in[0,T]\times\mathbb{R}_+:\ v\le \gamma(L-R_t) +\frac{m t_0}{T} \Big\}.
\]
{\bf\rm Step 3:} Simulate a Poisson point process $N$ on the state space $[0,T]\times\mathbb{R}_+$, independent of the Bessel process,  whose intensity measure is the Lebesgue one.\\[3pt]
{\bf\rm Step 4:} If $N(D_{R,T}^m)=0$ then set $Y=T$ otherwise go to Step 1.\\[5pt]
{\bf\rm Outcome:} the random variable $Y$.}
\end{framed}
It is obvious that the support of the random variable $Y$ is contained in $[0,t_0]$ by construction. 
\begin{prop}
\label{prop:exten-neg} Let $\gamma$ a function satisfying Assumptions \ref{assu-1}, \ref{assu-2} \& \ref{assu-5} and $G$ a standard gaussian r.v.  If $T$ has the same distribution as the conditional distribution of $(L-x)^2/G^2$ given $(L-x)^2/G^2\le t_0$, then the random variable $Y$ obtained by Algorithm (A3) has the same distribution as $\tau_L$ given $\tau_L\le t_0$. 
\end{prop}
\begin{proof}
\emph{Step 1.} Let us first describe the dependence of the distribution of the outcome $Y$ with respect to the distribution of the input $T$. Let us also note that, due to the condition on the function $\gamma$, $\gamma(L-R_t) +\frac{m t_0}{T}\ge 0$ for any $t$ and $T\le t_0$. 
We deduce that
\begin{align*}
\mathbb{P}( N(D_{R,T}^m)=0| R, T )&=\exp-\int_0^T\Big(\gamma(L-R_t) +\frac{m t_0}{T} \Big)\,\dint t\\
&=e^{-mt_0}\exp-\int_0^T\gamma(L-R_t)\,\dint t. 
\end{align*}
Following the arguments developed in Theorem \ref{thm:dens}, we obtain,  for any non-negative or bounded function  $\psi$,
\begin{align}\label{eq:linkY}
\mathbb{E}[\psi(Y)]&=\frac{1}{\mathbb{P}(N(D_{R,T}^m)=0)}\mathbb{E}\Big[ \psi(T)\mathbb{P}(N(D_{R,T}^m)=0|T)\Big]\nonumber \\
&=\frac{\mathbb{E}[\psi(T) e^{-mt_0}\eta(T)]}{\mathbb{E}[e^{-mt_0}\eta(T)]}=\frac{\mathbb{E}[\psi(T)\eta(T)]}{\mathbb{E}[\eta(T)]}.
%&=\mathbb{E}\Big[ \psi(T^{(1)})\frac{1_{\overline{E}_1}}{\mathbb{P}(\overline{E}_1)}\Big],
\end{align}
where $\eta$ is defined in \eqref{prop:eq:2}.
In other words, the \emph{p.d.f.} of the outcome $Y$ is given by
\[
f_Y(t)=\frac{1}{\Xi}\,\eta(t)f_T(t), 
\]
where $\Xi$ is the normalization coefficient. The result is quite similar to the statement of Theorem \ref{thm:dens}. The differences are the following: on one hand, the function $\gamma$ can take negative values, on the other hand the support of the distribution of $T$ has to be compact.\\
\emph{Step 2.} Let us now link the conditional distributions of $Y$ and $\tau_L$. Girsanov's transformation, see Proposition \ref{prop:Girsanov}, implies
\begin{equation}\label{eq:numer}
\mathbb{E}_\mathbb{P}[\psi(\tau_L)1_{\{ \tau_L\le t_0 \}}]=\mathbb{E}_\mathbb{Q}[\psi(\tau_L)1_{\{ \tau_L\le t_0 \}}\eta(\tau_L)]\exp \Big\{\beta(L)-\beta(x)\Big\},
\end{equation}
where $\mathbb{P}$ (resp. $\mathbb{Q}$) corresponds to the distribution of the diffusion $X$ (resp. the Brownian motion $B$). The particular case $\psi \equiv 1$ leads to:
\begin{equation}\label{eq:denom}
\mathbb{P}(\tau_L\le t_0)=\mathbb{E}_\mathbb{Q}[1_{\{ \tau_L\le t_0 \}}\eta(\tau_L)]\exp \Big\{\beta(L)-\beta(x)\Big\}.
\end{equation}
Combining the ratio between \eqref{eq:numer} and \eqref{eq:denom} on one side and the equation \eqref{eq:linkY} on the other side leads to 
\begin{align*}
\mathbb{E}_{\mathbb{P}}[\psi(\tau_L)|\tau_L\le t_0]&=\frac{\mathbb{E}_\mathbb{Q}[\psi(\tau_L)1_{\{ \tau_L\le t_0 \}}\eta(\tau_L)]}{\mathbb{E}_\mathbb{Q}[1_{\{ \tau_L\le t_0 \}}\eta(\tau_L)]}=\frac{\mathbb{E}_\mathbb{Q}[\psi(\tau_L)\eta(\tau_L)| \tau_L\le t_0 ]}{\mathbb{E}_\mathbb{Q}[\eta(\tau_L)|\tau_L\le t_0 ]}\\
&=\frac{\mathbb{E}[\psi(T)\eta(T)]}{\mathbb{E}[\eta(T)]}=\mathbb{E}[\psi(Y)],
\end{align*}
where $T$ has the same law as $\tau_L|\tau_L\le t_0$ under the probability measure $\mathbb{Q}$ (Brownian hitting time). In order to conclude, it suffices to note that, under $\mathbb{Q}$, the random variable $\tau_L$ and $(L-x)^2/G^2$ are identically distributed.
\end{proof}
Of course, for simulation purposes, we don't use directly Algorithm (A3) since it is impossible to simulate the Poisson point process $N$ on the whole unbounded state space $[0,T]\times\mathbb{R}_+$. So we apply exactly the same procedures developed in Section \ref{sec:1} which transform Algorithm (A0) into (A1) or (A2). Such procedures modify Algorithm (A3) and permit to easily execute numerical simulations (Section \ref{sec:num:cond}).

\section{Unbounded drift terms}
\label{sec:unbounded}
The aim of this section is to handle with drift terms $b$ which do not satisfy the condition $\gamma(x)\le \kappa$ for any $x\le L$. From now on, we shall just assume that
\begin{equation}
\label{eq:cond:light}
\gamma(x)\ge 0,\quad \forall x\le L.
\end{equation}
The algorithms (A1) and (A2) cannot be used in such a situation since the domain associated to the Poisson point process is unbounded. One way to overcome this difficulty is to replace the drift term of the diffusion by a modified version: for a given large parameter $\rho$, 
\begin{equation}\label{eq:mod:drift}
b_\rho(x)=\left\{\begin{array}{ll}
b(x) &\mbox{if}\ -\rho\le x\le L\\
b(-\rho)+b'(-\rho)(x+\rho)e^{x+\rho}&\mbox{if}\ x<-\rho.
\end{array}
\right.
\end{equation}
In other words, it suffices to keep the function $b$ on the interval $[-\rho,L]$ and to extend this function in a regular way to $]-\infty,L]$ in order to get a bounded function. Here $\lim_{x\to-\infty}b_\rho(x)=b(-\rho)$, the boundedness being verified.

Obviously $b_\rho$ is a $\mathcal{C}^1$-continuous function and both Algorithm (A1) and (A2) can be applied to the diffusion associated to this new drift term. The diffusion is denoted $X^\rho$. %We can compute the associated function $\gamma$
%\[
%\gamma_\rho:=\frac{b_\rho^2+b_\rho'}{2}.
%\]
%Of course such a procedure does not lead to an exact simulation of $\tau_L$ but to an approximation of it denoted by $\tau_L^\rho$. Let us denote by Algorithm ${\rm (A1)}^\rho$ this modified algorithm ($\gamma$ replaced by $\gamma_\rho$ in Algorithm (A1)) and let us study this approximation.  We introduce the following function
%\[
%p(x):=\int_0^x e^{-2\beta(y)}\,dy \quad \mbox{with}\quad \beta(y)=\int_0^yb(r)\,dr.
%\]
\begin{prop}\label{prop:unbounded} Let the function $\gamma$ satisfy Assumption \ref{assu-1}. Moreover we assume that $\lim_{x\to-\infty}p(x)=-\infty$ where $p$ is defined in \eqref{eq:def:3}. Then the distribution of $\tau_L^\rho$ converges towards the distribution of $\tau_L$ as $\rho\to\infty$. Moreover the Kolmorogov distance satisfies:
\[
d(\tau_L,\tau_L^\rho):=\sup\Big\{|F_{\tau_L^\rho}(t)-F_{\tau_L}(t)   |\ : \ t\in\mathbb{R}_+\Big\}=\mathcal{O}\Big(-p(-\rho)\Big)\ \mbox{as}\ \rho\to\infty. 
\]
Here $F_Y(x)$ stands for the cumulative distribution function of the r.v. $Y$.
\end{prop}
In practice, in order to approximate $\tau_L$ we simulate $\tau_L^\rho$ using one of the algorithms presented before, with the parameter $\rho$ large enough. Let us notice that the condition $\lim_{x\to-\infty}p(x)=-\infty$ appearing in the statement of Proposition \ref{prop:unbounded} implies Assumption \ref{assu-2} as mentioned in Remark \ref{rem}.
\begin{proof}
The arguments are quite straightforward: let us consider $x\in\mathbb{R}_+$, then
\begin{align*}
F_{\tau_L}(t)&=\mathbb{P}(\tau_L\le t,\, \tau_L\le \tau_{-\rho})+\mathbb{P}(\tau_L\le t,\, \tau_L> \tau_{-\rho}).
\end{align*}
Since $X^\rho$ and $X$ have the same distribution before exiting from the interval $[-\rho,L]$ (same drift term on this space subset), we obtain
\begin{align*}
|F_{\tau_L}(t)-F_{\tau_L^\rho}(t)|&=|\mathbb{P}(\tau_L\le t,\, \tau_L> \tau_{-\rho})-\mathbb{P}(\tau_L^\rho\le t,\, \tau_L^\rho> \tau_{-\rho}^\rho)|\\
&\le 2\, \mathbb{P}(\tau_L> \tau_{-\rho})=2\, \frac{p(L)-p(x)}{p(L)-p(-\rho)}\sim -\frac{2}{p(-\rho)}.
\end{align*}
\end{proof}
The efficiency of this algorithm is already described in Section \ref{sec:efficiency} and numerical results are presented in Section \ref{sec:num-unbouned}.

\section{Examples and numerics}
\label{sec:numerics}
In this section we study the performances of the different algorithms  proposed through the paper in order to simulate $\tau_L$ the first-passage time of the diffusion starting in $x$ with $x<L$. Let us just give a quick review in the following table describing which algorithm should be used when the function $\gamma$ satisfies particular bounds on the interval $]-\infty,L]$.\\[5pt]
\renewcommand{\arraystretch}{1.5}
\centerline{\begin{tabular}{|c|c|c|}
\hline
Condition on $\gamma$ & r.v. simulated & Algorithm and statement\\
\hline
\hline
$0\le \gamma(x)\le \kappa$ & $\tau_L$ &  (A1) or (A2) in Proposition \ref{prop:A1:conv}\\
\hline
$0<\gamma_0\le \gamma(x)\le \kappa$ &  $\tau_L$ & ${\rm (A1)_{shift}}$ or ${\rm (A2)_{shift}}$ in Prop. \ref{prop:inv:gauss} \\
\hline
$-m\le  \gamma(x)\le \kappa$ & $\tau_L$ given $\tau_L\le t_0$ &   (A3) in Prop. \ref{prop:exten-neg}\\
\hline
$0 \le \gamma(x)$ & $\tau_L^\rho$ (approx.) &   ${\rm (A1)^\rho}$ or ${\rm (A2)^\rho}$ in Prop. \ref{prop:unbounded}\\
\hline
\end{tabular}}\\[7pt]
Let us recall that the efficiency of all these algorithms can be increased using the \emph{space splitting} argument proposed in Section \ref{sec:number-ite}.

 In order to illustrate results stem from simulations, we consider the first hitting time of a Brownian motion with drift $\mu$. In this case, a closed form of the \emph{p.d.f.} for $\tau_L$ is available (see, for instance \cite{K-S} p. 197):
\[
\mathbb{P}(\tau_L\in \dint t)=\frac{L-x}{\sqrt{2\pi t^3}}\exp-\frac{(L-x-\mu t)^2}{2t}\, \dint t,\quad t> 0.
\] 
\noindent \begin{minipage}{6cm}
Since the level $L$ is larger than the starting position $x$, Assumption \ref{assu-2} will be satisfied if and only if $\mu>0$. \\

In the opposite figure, there is a comparison between the theoretical probability density function of $\tau_L$ and its normalized histogram ($10\, 000$ simulations) obtained by Algorithm (A1).\\

Here $b(x)\equiv 1$, $x=0$ and level $L=2$.
\end{minipage}\hfill
\begin{minipage}{7cm}
\centering
\includegraphics[width=7cm]{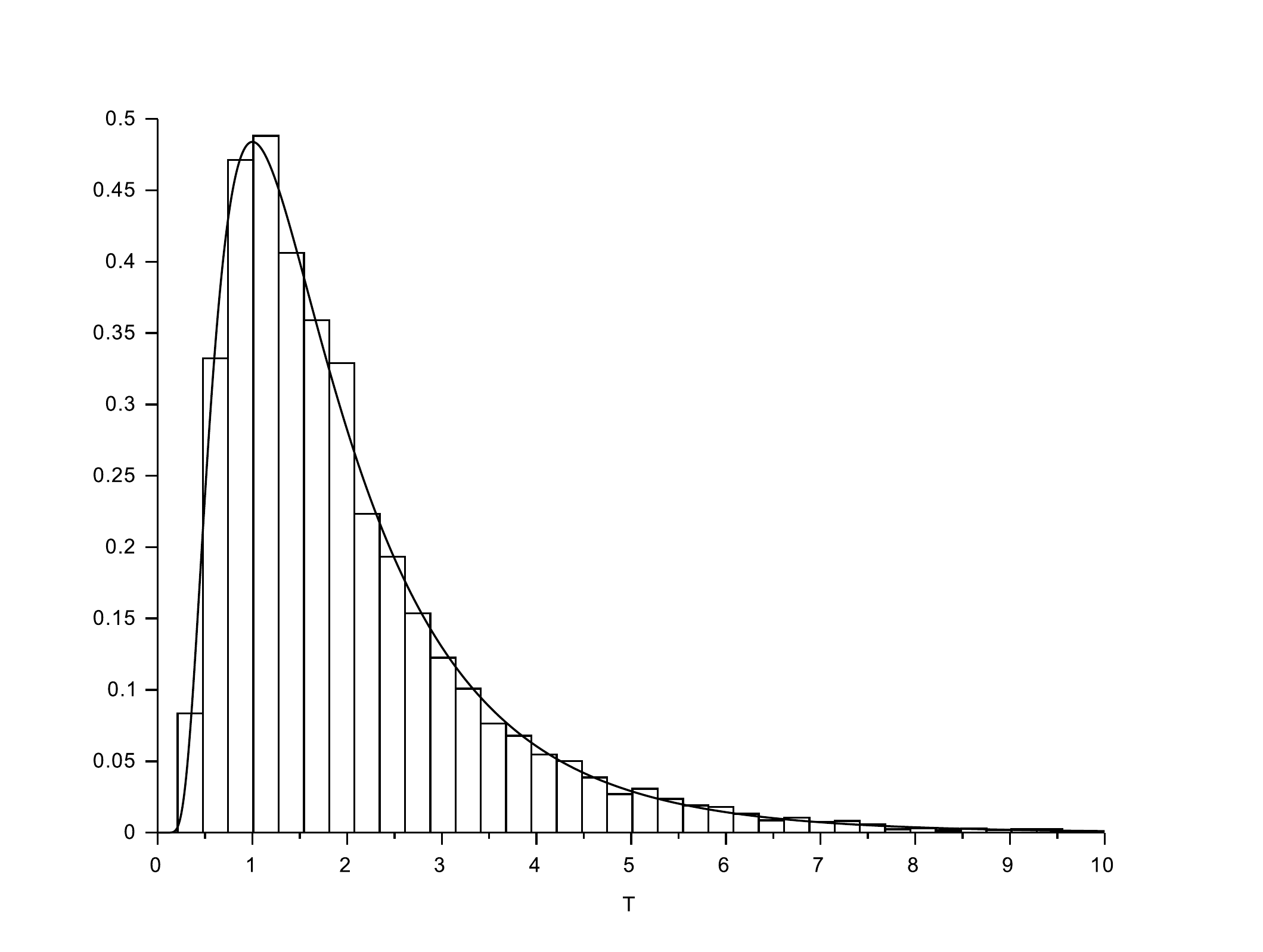}
\end{minipage}
 
\subsection{Algorithms (A1) and (A2)}
We shall compare Algorithm (A1) and (A2), both of them simulating exactly the first-passage time $\tau_L$ for a diffusion process satisfying Assumption \ref{assu-1}--\ref{assu-3}.\\[5pt]
{\bf Example 1.}
We consider the following stochastic differential equation: %already used in \cite{beskos2005exact} in order to obtain numerics:
\begin{equation}\label{eq:ex1}
dX_t=(2+\sin(X_t))\,dt+dB_t,\quad t\ge 0, \quad X_0=0.
\end{equation}
We first observe that $\gamma(x)=(b^2(x)+b'(x))/2=((2+\sin(x))^2+\cos(x))/2$ satisfies $0\le \gamma\le 5$. We use Algorithm (A1) in order to produce an histogram of the hitting time distribution (Figure \ref{Fig:1}).
\begin{figure}[h]
\includegraphics[width=6cm]{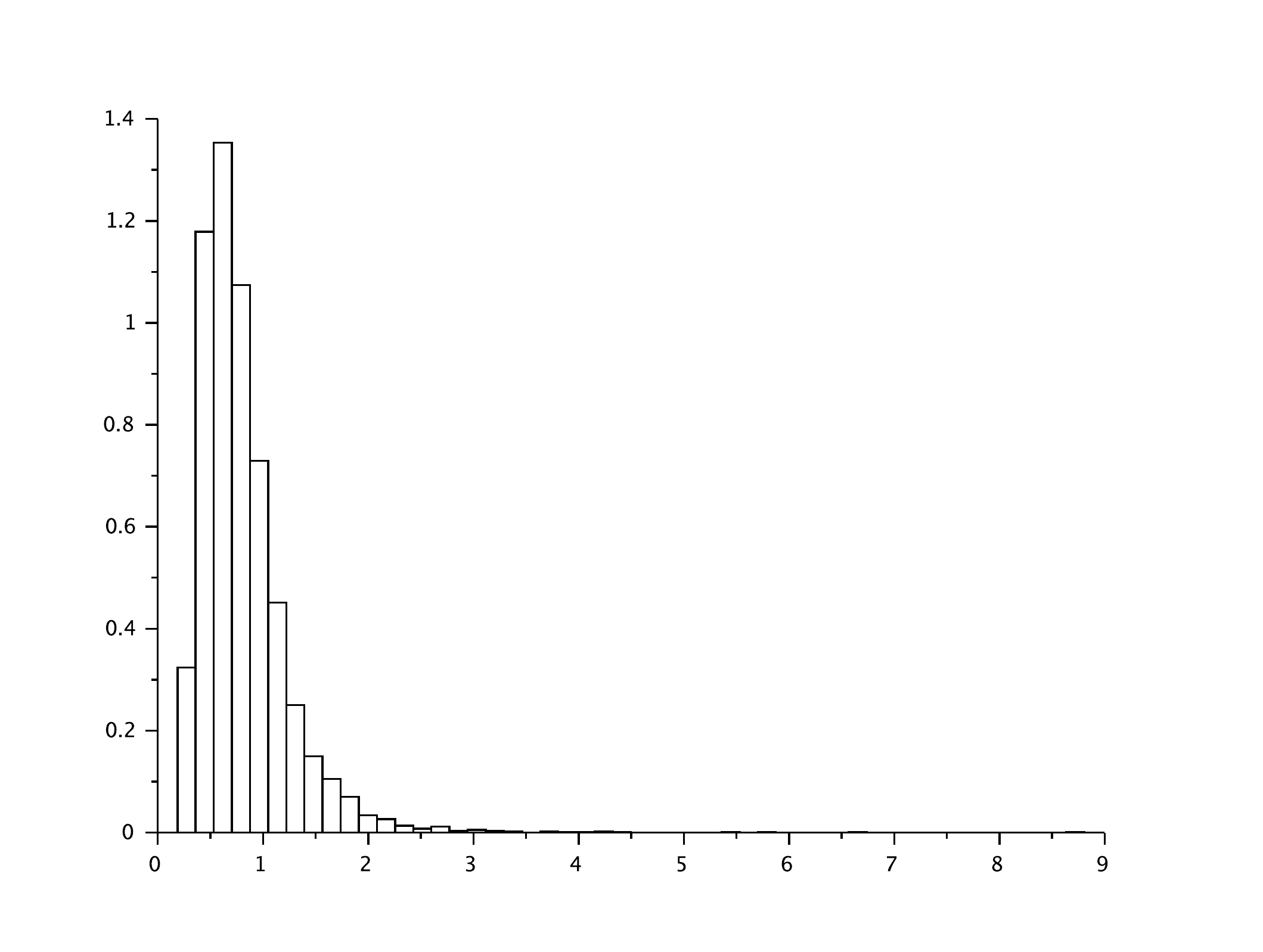}
\hspace*{0.1cm}
\includegraphics[width=6cm]{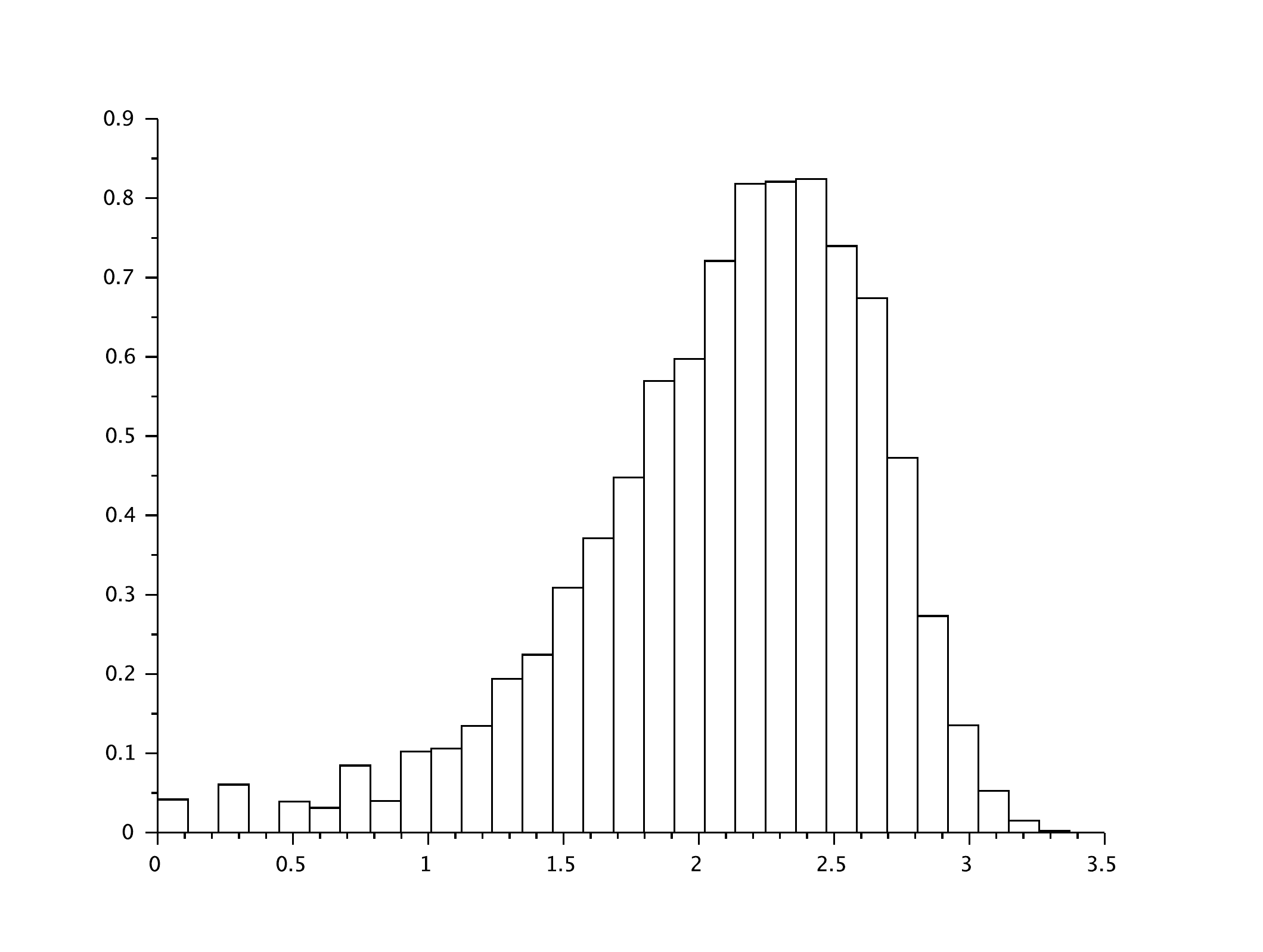}
\caption{Histogram of the hitting time distribution for $10\,000$ simulations corresponding to the level $L=2$ and starting position $X_0=0$ (left), histogram of the number of iterations in Algorithm (A1) in the $\log_{10}$-scale (right).}\label{Fig:1}
\end{figure}
The total number of random variables used in this algorithm is of course random and the associated histogram (Figure \ref{Fig:2}) illustrates its distribution. 

 \begin{figure}[ht]
\includegraphics[width=6cm]{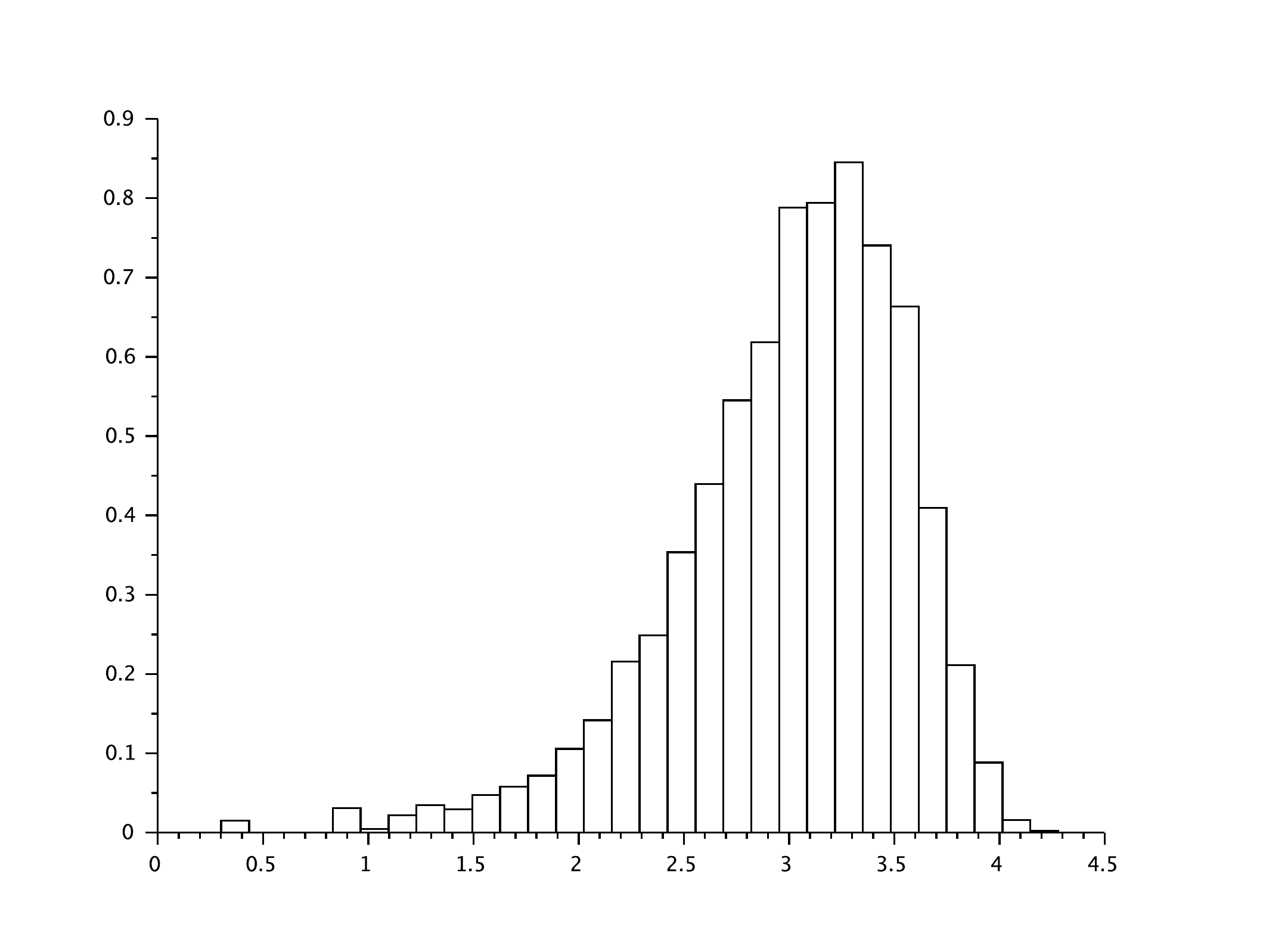}
\hfill
\includegraphics[width=6cm]{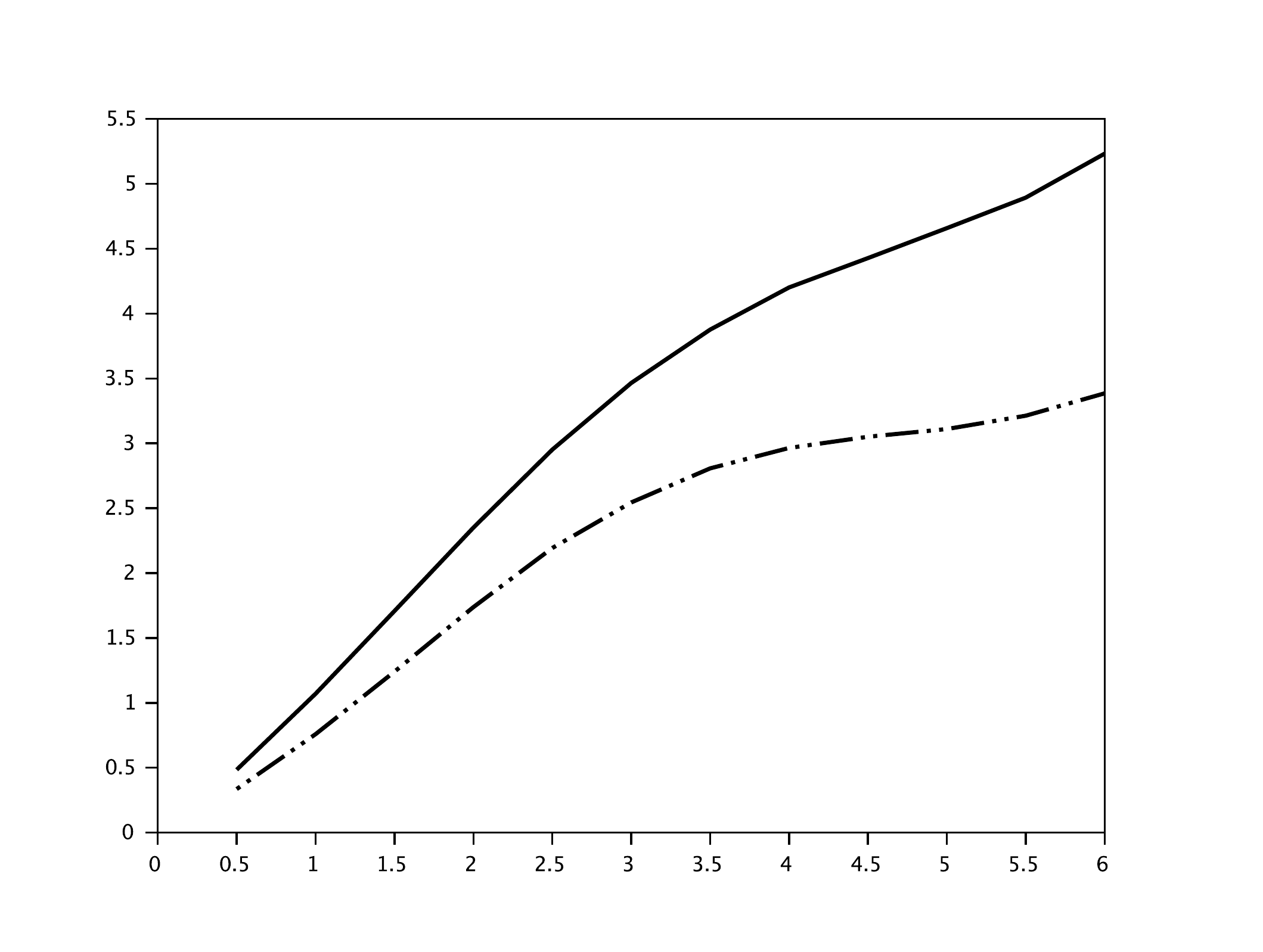}
\caption{Number of random variables used in Algorithm (A1) for $10\,000$ simulations with $L=2$, $X_0=0$ in the $\log_{10}$-scale (left) and mean number of iterations versus the level height $L$ for Algorithm ${\rm (A1)_{shift}}$ respectively \emph{(A1)} (dashed line resp. solid line), both curves are in the $\log_{10}$-scale ($10\,000$ simulations have been used for the average estimation).}
\label{Fig:2}
\end{figure}

Let us just note that Algorithm (A1) is quite time consuming since simulating $10\,000$ hitting times (corresponding to $L=2$ and $X_0=0$) requires about $300$ cpu. This time increases very fast as $L$ becomes large: the figure (right) presents the number of iterations versus $L$. In order to increase the algorithm efficiency and consequently reduce the computation time, one way is to shift the function $\gamma$ as described in Proposition \ref{prop:inv:gauss}. The function $\gamma$ introduced in \eqref{eq:ex1} satisfies \(\gamma(x)\ge \gamma_0:=1/4$.
%%Indeed the minimum of $\gamma$ is reached for $x_0$ satisfying $\cos(x_0)=\frac{\sin(x_0)}{4+2\sin(x_0)}$. Hence $\gamma(x)\ge \gamma(x_0)\ge 1/4$.  
The comparison of the efficiency of Algorithm $(A1)$ (solid curve) and Algorithm ${\rm (A1)_{shift}}$ (dashed curve) is presented in Figure \ref{Fig:2}.

\begin{figure}[h]
\includegraphics[width=6cm]{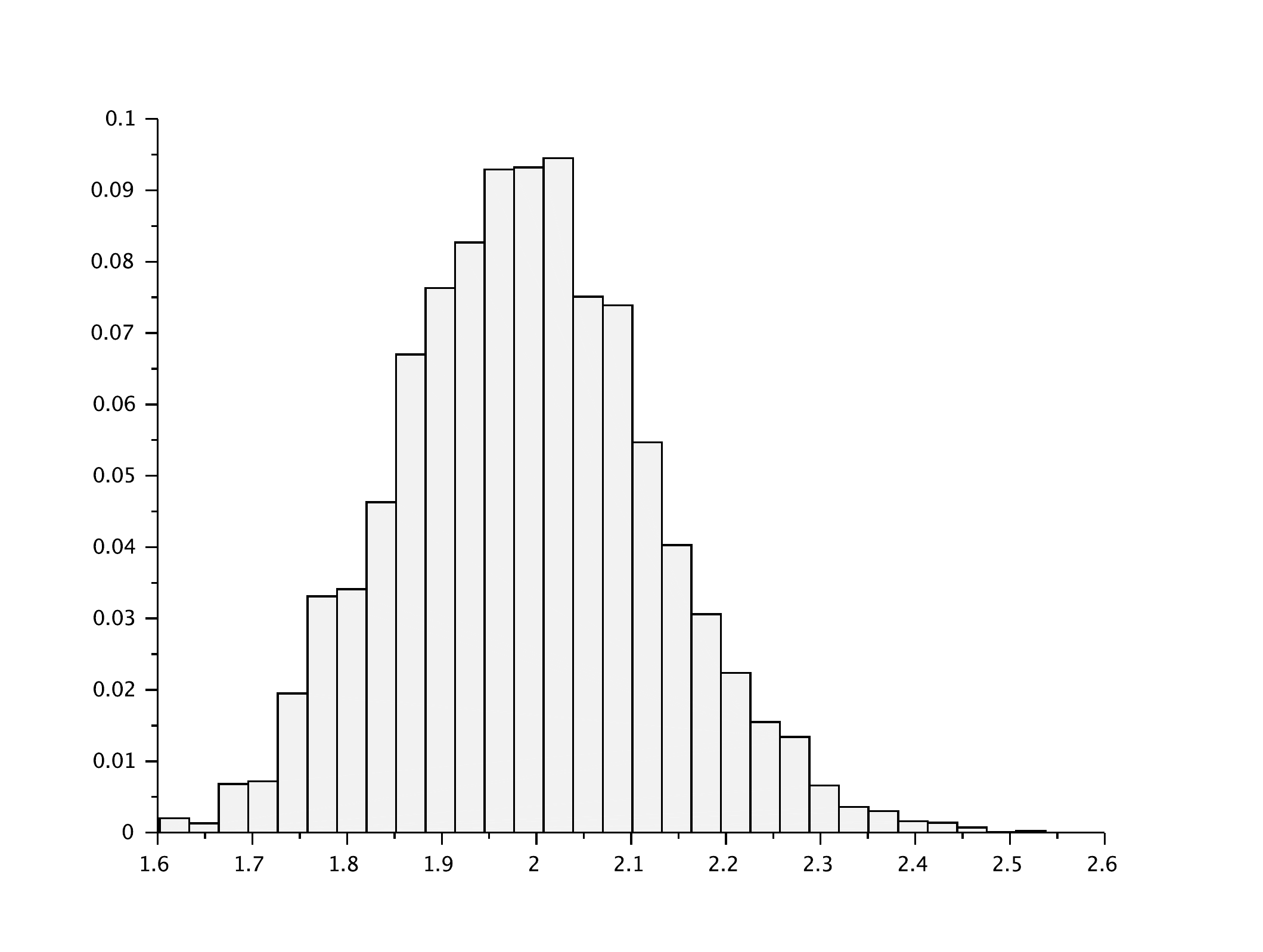}
\hfill
\includegraphics[width=6cm]{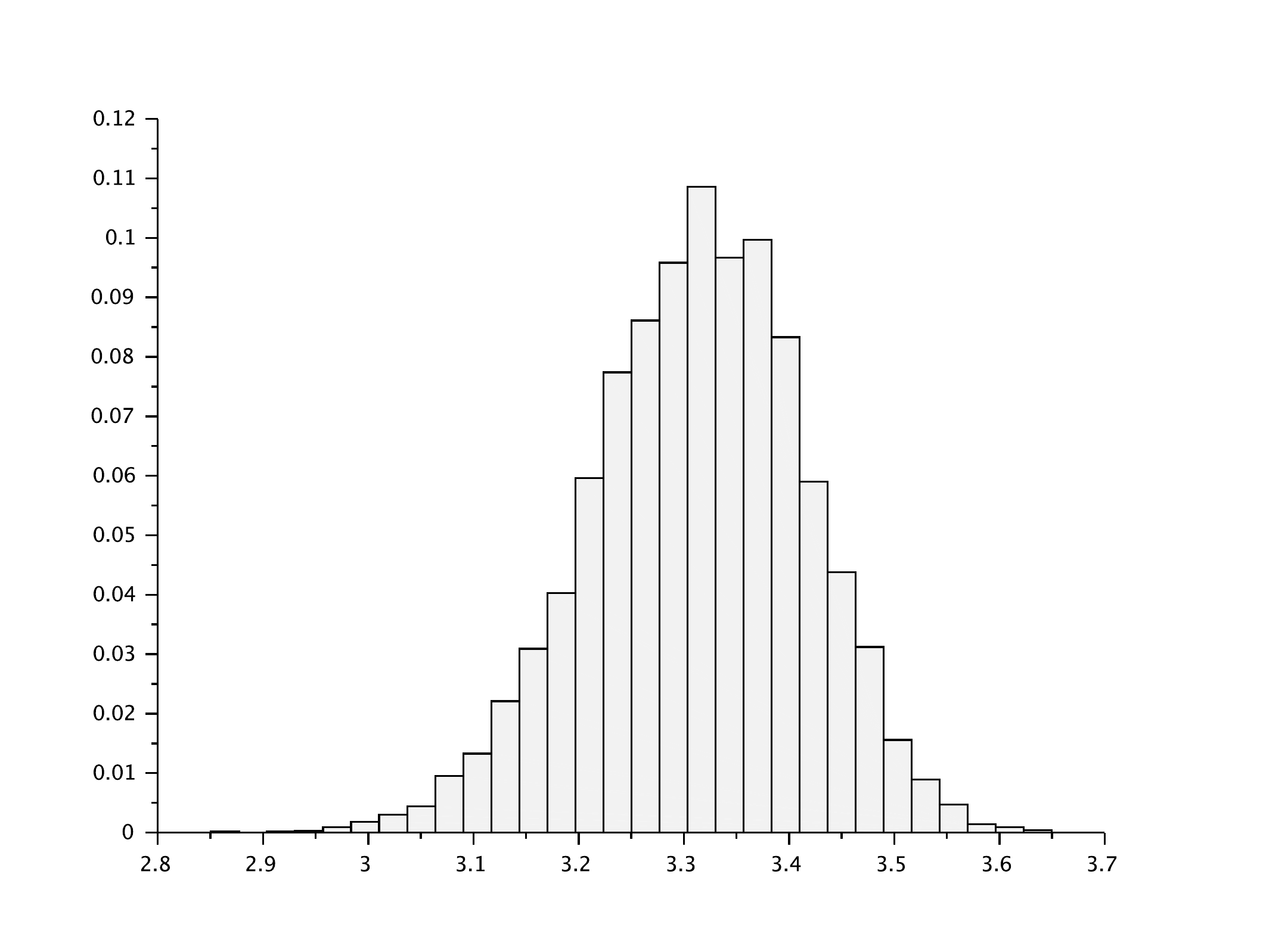}
%{histogram_Nb_iterations-10000-2-0.pdf}
\caption{Histogram of the number of random variables in Algorithm (A1) using space splitting for $10\,000$ with $L=2$, $X_0=0$, $k=20$ (left), $L=20$, $k=20$ (right) both in the $\log_{10}$-scale.}\label{Fig:4}
\end{figure}
The example \eqref{eq:ex1} also permits to illustrate the second modification of the algorithm developed in Section \ref{sec:number-ite}: the \emph{space splitting} method. Let us recall that we can replace the simulation of $\tau_L$ for a diffusion starting in $x$  (denoted here by $\tau(x\to L)$) by the simulation of $k$ independent first-passage times $\tau(x+(i-1)(L-x)/k\to x+i(L-x)/k)$ for $i=1,\ldots,k$. It suffices then to take the sum of all these quantities. 
Figure \ref{Fig:4} (left) represents the number of random variables for the particular case: $L=2$, $X_0=0$ and $k=20$. We can notice that such a procedure is slightly more efficient for this particular case: the average number of random variables (estimation with $10\,000$ simulations) namely goes from about $1791$ down to $102$. It permits also to deal with large level values (Figure \ref{Fig:4} -- right) which is a huge task with Algorithm (A1). For each level $L$ there is an optimal splitting (optimal value $k$) illustrates by Figure \ref{Fig:4bis}.

\begin{figure}[h]
%\centerline{\includegraphics[width=6cm]{Figures/FigModversusk-Nbrv-Ex1-Lev5-start0-N10000.pdf}}
\centerline{\includegraphics[width=6cm]{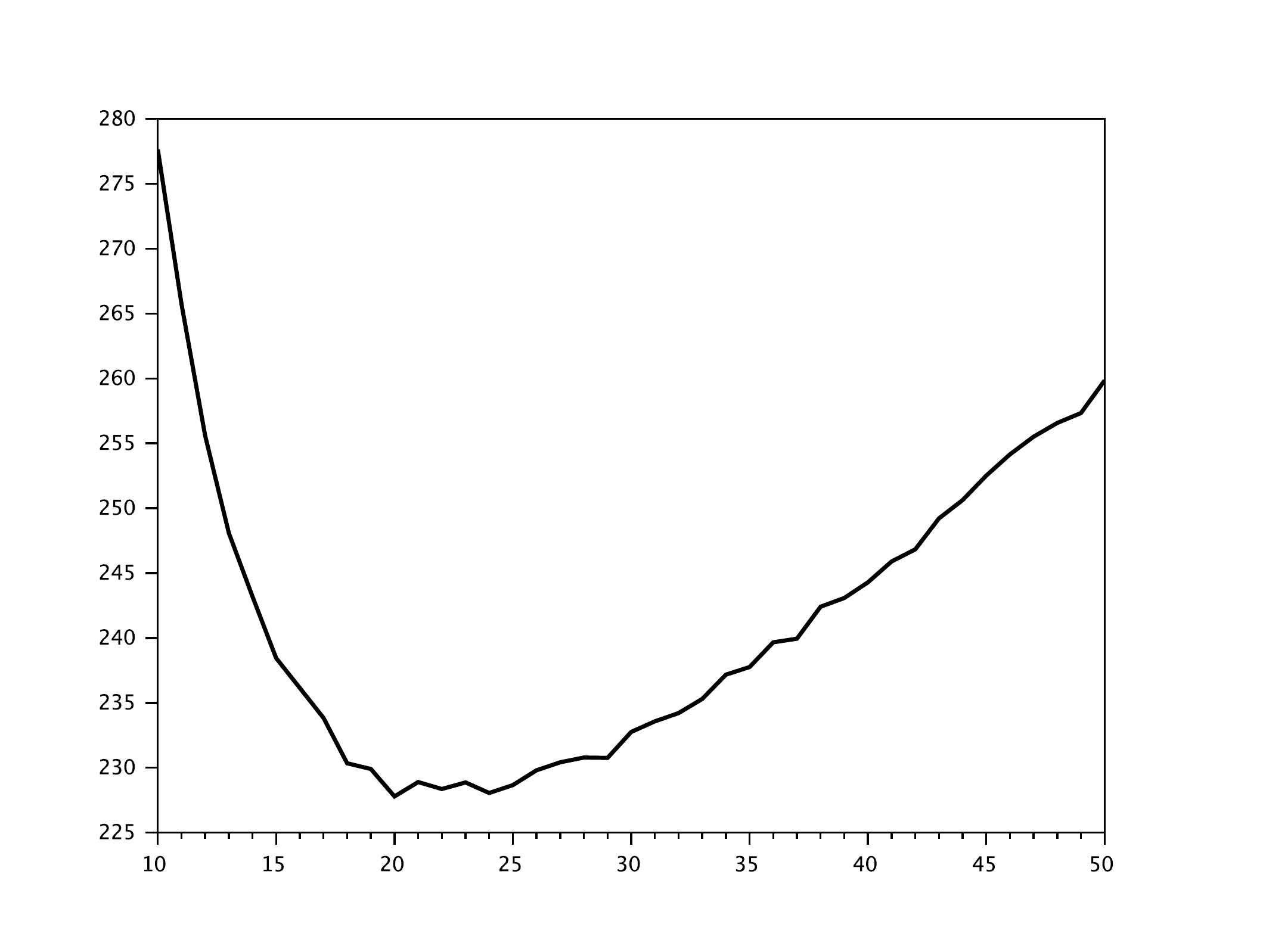}}
%\hspace*{0.1cm}
%\includegraphics[width=6cm]{}
%{histogram_Nb_iterations-10000-2-0.pdf}
\caption{Averaged number of random variables used in Algorithm (A1) versus the number of slices $k$ with $X_0=0$ and $L=5$. The averaging uses $10\,000$ simulations.}\label{Fig:4bis}
\end{figure}

Up to this point, all simulations use Algorithm (A1). Of course, it is also possible to deal with Algorithm (A2) in  such a situation. We obtain the same  hitting time distribution and can therefore compare the number of iterations and the number of random variables used. Let us notice that both Algorithms (A1) and (A2) are based on the simulation of a Poisson point process defined on the rectangle $[0,T]\times[0,\kappa]$: for one algorithm, the simulation starts with the point with smallest abscissa and then goes from one point to the next one by increasing the abscissa, for the other one, the order of simulation depends on the ordinate. It is possible of course to have a one-to-one correspondance between these two simulation technics: each point $(t,x)\in[0,T]\times[0,\kappa]$ of the Poisson process is then transformed as follows
\[
(t,x)\mapsto(xT/\kappa,t\kappa/T).
\]
Consequently the same sequence of random points can be used for both Algorithms and it becomes easy to compare the total number of random points. Let us introduce $\mathcal{N}_1$ the total number of random variables used by  Algorithm (A1) and respectively $\mathcal{N}_2$ by Algorithm (A2), then we define the difference
\begin{equation}
\label{eq:def:diff}
\Delta:=\mathcal{N}_2-\mathcal{N}_1.
\end{equation}
In the diffusion case \eqref{eq:ex1}, the estimated mean $\overline{\Delta}$ is equal to 
$-196.0$ and its standard deviation equals $1853.2$ when observing $10\,000$ simulations with $x=0$ and $L=2$. We claim therefore that Algorithm (A2) is more efficient than Algorithm (A1) for Example 1 since $\overline{\Delta}$ is significantly negative (p-value: $1.9E-26$). However this difference is small in comparison to the mean number of random variables: $\overline{\Delta}/\overline{\mathcal{N}}_1\approx - 0.11$.  In other situations, this difference can be much stronger or positive. It essentially depends on the drift term $b$ of the diffusion.  
Let us consider a second situation.

\noindent {\bf Example 2.} Let us consider the following SDE:
\begin{equation}\label{eq:ex2}
dX_t=(1+\arctan(1-X_t))\,dt+dB_t,\quad t\ge 0 \quad X_0=0,
\end{equation}
associated to the first-passage time for $L=1$. Let us note that $\tau_1<\infty$ a.s.

\begin{figure}[h]
\includegraphics[width=6cm]{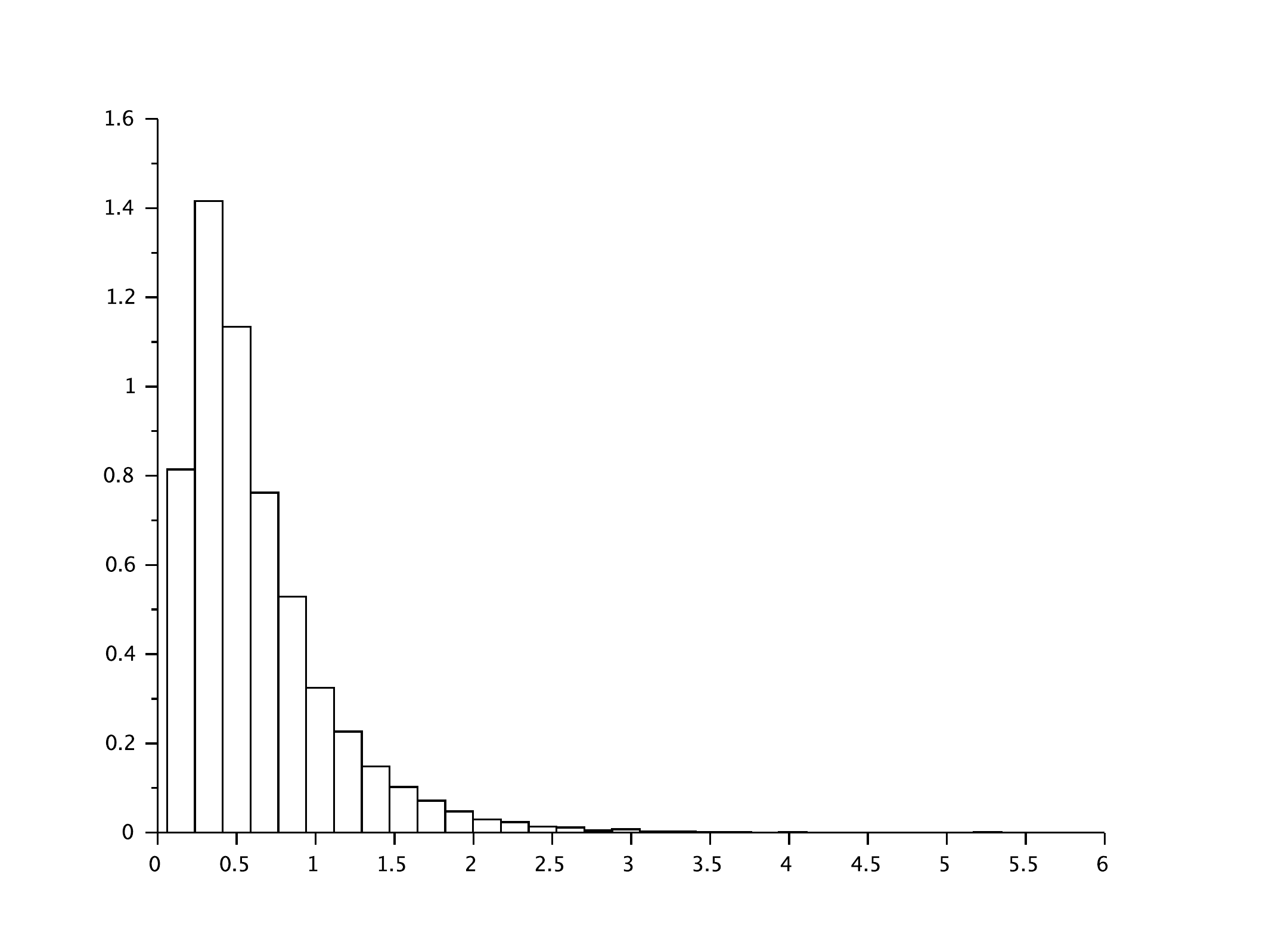}
%{histogram_hitting_time-10000-2-0.pdf}
\hfill
\includegraphics[width=6cm]{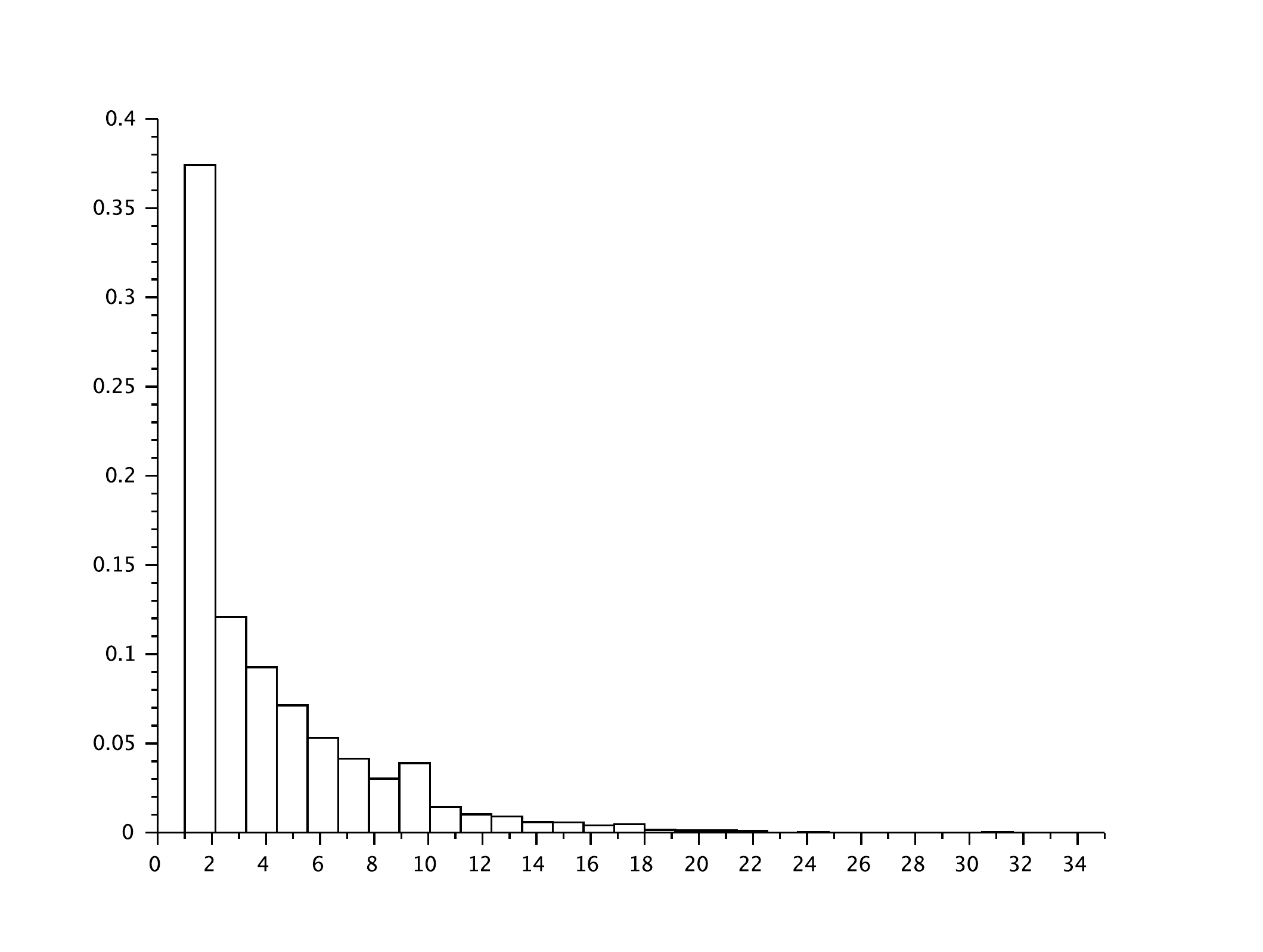}
\caption{Histogram of the hitting time distribution using Algorithm (A1) (left) and histogram of the number of iterations (right) for $X_0=0$, $L=2$ and $10\,000$ simulations.}\label{Fig:ex2_T}
\end{figure}

We observe that 
$\gamma(x)=\frac{1}{2}\left((1+\arctan(1-x))^2-\frac{1}{1+(1-x)^2}\right)$ satisfies $0\le \gamma\le \frac{1}{2}(1+\pi/2)^2$.
%We observe that $\gamma(x)=\frac{1}{2}\left(\arctan(x)^2+\frac{1}{1+x^2}\right)$ satisfies $0\le \gamma\le 1.25$.
Both Algorithms (A1) and (A2) can be used in order to simulate the first-passage time $\tau_L$ and the same kind of study is achieved: here, the estimated difference  $\overline{\Delta}=-169.7$, its standard deviation $235.2$ and finally the reduction rate $\overline{\Delta}/\overline{\mathcal{N}}_1\approx - 0.67$ for $x=0$, $L=1$ and $10\,000$ simulations. The choice between $(A1)$ and $(A2)$, in such a situation, becomes an important question (the hitting time distribution is illustrated by Figure \ref{Fig:ex2_T}).

% mean(Nb_rv2) ans  = 84.3818
% mean(Nb_rv1) ans  = 254.0572
% Delta=Nb_rv2-Nb_rv1;
% Delta_average=mean(Delta) Delta_average  = -169.6754
% Delta_stdev=stdev(Delta) Delta_stdev  =  235.17652
% Delta_average*sqrt(10000)/Delta_stdev ans  =  -72.148104
% Delta_average/mean(Nb_rv1) ans  = -0.667863

\mathversion{bold}
\subsection{Simulation of $\tau_L$ given $\tau_L\le  t_0.$}
\label{sec:num:cond}
\mathversion{normal}
In this section, we illustrate Algorithm (A3) presented in Section \ref{sec:before} which permits to overcome the condition $\gamma \ge 0$. Indeed this condition can be weaken (Assumption \ref{assu-5}) provided the study only concerns  the distribution of the first-passage time $\tau_L$ conditioned on the event $\tau_L\le t_0$. Here $t_0$ is any fixed time.

We therefore introduce the following SDE:
\begin{equation}
\label{eq:exbefore}
dX_t=-\arctan(X_t)\,dt+dB_t,\quad 0\le t\le t_0.
\end{equation}
We focus our attention to the case $X_0=0$, $L=1$ and different values for $t_0$. Let us observe that $\gamma(x)=(\arctan(x)^2-1/(1+x^2))/2$ satisfies $-1/2\le \gamma(x)\le \pi^2/8$ and that the first-passage time is almost surely finite (Assumption \ref{assu-2}). As described in the statement of Proposition \ref{prop:exten-neg}, the simulation of $\tau_L$ given $\tau_L\le t_0$ is based on a two-steps acceptance/rejection method: 
\begin{itemize}
\item In order to simulate $G^2$ given $G^2\ge (L-x)^2/t_0$ (here $G$ stands for a standard gaussian variable), we use shifted exponentially distributed random variables with an optimal (classical) rejection rule.
\item Using the acceptance/rejection procedure introduced in Algorithm (A3) we transform the distribution of $(L-x)^2/G^2$ given $(L-x)^2/G^2\le t_0$ into $\tau_L$ given $\tau_L\le t_0$.
\end{itemize}
\begin{figure}[h]
\includegraphics[width=6cm]{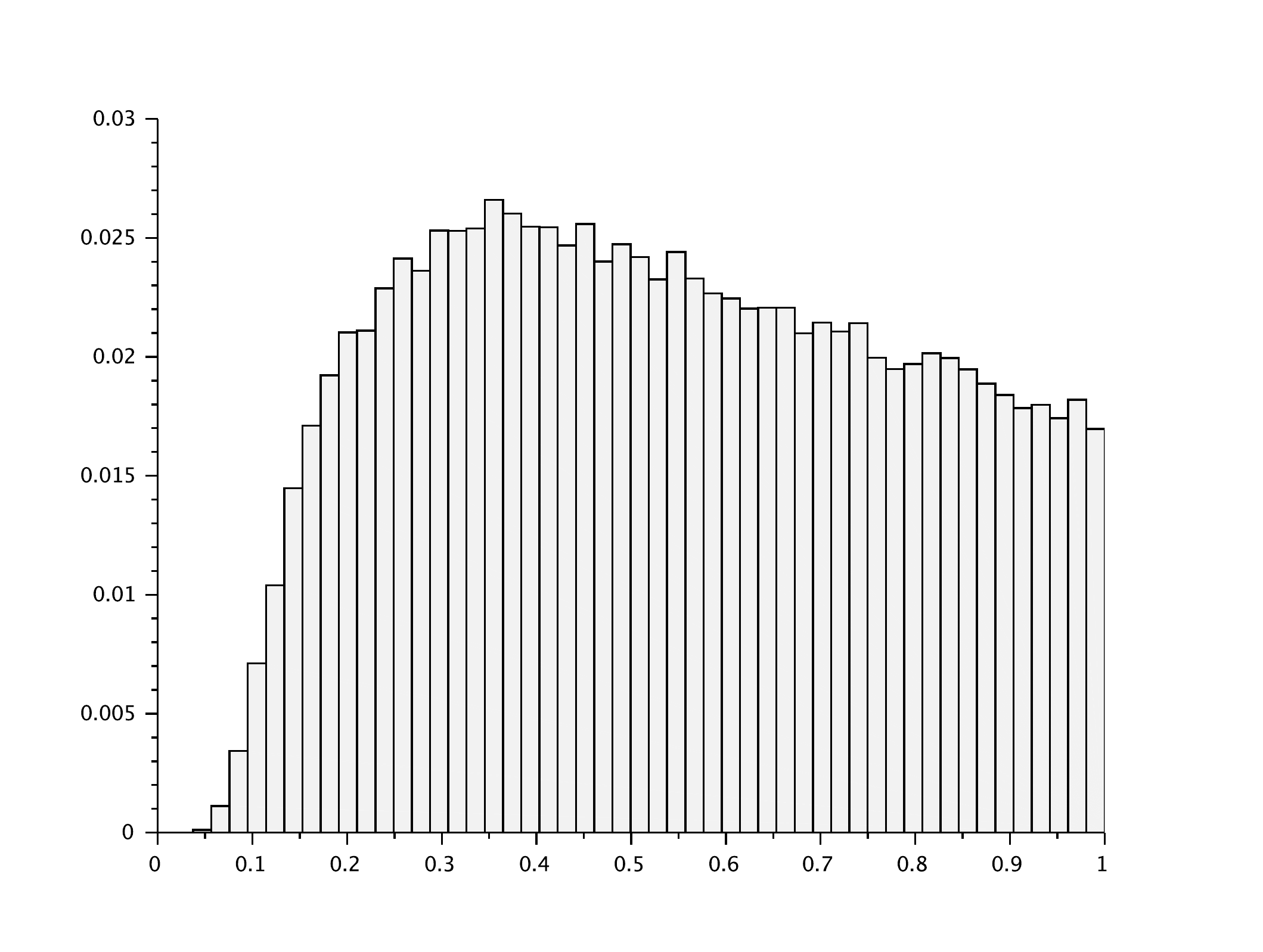}
%{histogram_hitting_time-10000-2-0.pdf}
\hfill
\includegraphics[width=6cm]{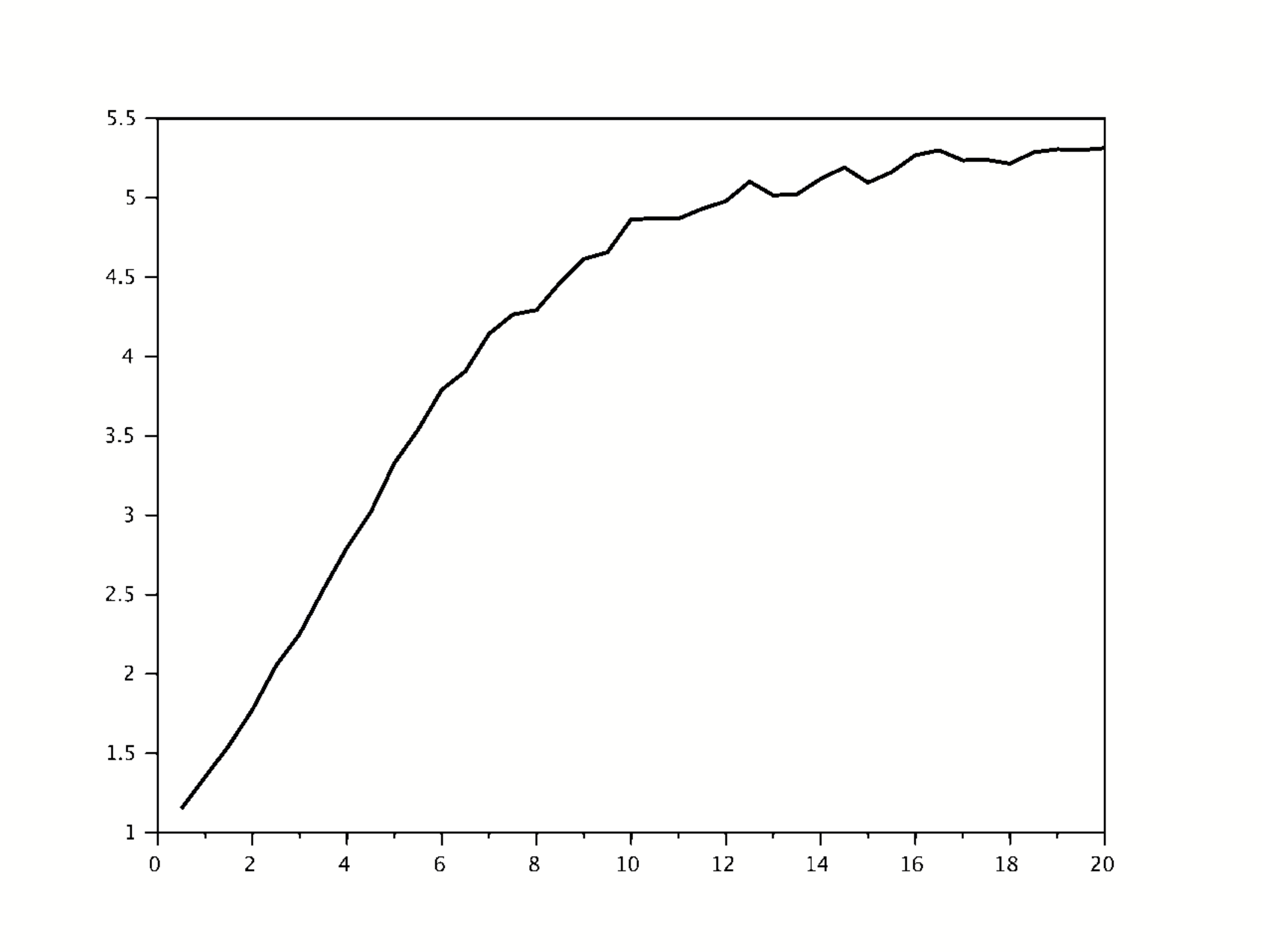}
\caption{Histogram of the hitting time distribution using Algorithm (A3) for $t_0=1$  and $100\,000$ simulations (left) and averaged number of iterations in Algorithm (A3) versus $t_0$ (right) for $X_0=0$, $L=1$ and $10\,000$ simulations.}\label{Fig:ex2_Tbis}
\end{figure}
In this illustration, we choose for the Poisson point process appearing in Algorithm (A3) the same kind of simulation as Algorithm (A1). Of course, it is also possible to adapt the procedure in order to follow the technics of Algorithm (A2) and finally to compare the efficiency of both Algorithms (we let this task to the reader).

\mathversion{bold}
\subsection{Modified Algorithm ${\rm (A1)}^\rho$}\label{sec:num-unbouned}
\mathversion{normal}

In this section we consider diffusion processes $(X_t,\ t\ge 0)$ satisfying \eqref{eq:0} with unbounded drift terms $b$. We apply the modified algorithm described in Section \ref{sec:unbounded} in order to approximate the first-passage time through the level $L$.  Let us consider as example an Ornstein-Uhlenbeck process with $b(x)=-\alpha x+\beta$ where $\alpha$ and $\beta$ are positive constants. We compute easily 
\begin{equation}
\gamma(x)=\frac{1}{2}(-\alpha x+\beta)^2-\frac{\alpha}{2},
\end{equation}
and observe that in general $\gamma$
is neither a bounded function on the whole interval $]-\infty,L]$ nor a positive one. The following choice of parameters $\alpha=0.3$, $\beta=1$ with starting position $X_0=0$ and boundary $L=1$ ensures that $\gamma$ is a positive function but $b$ remains unbounded, that's why we replace the original drift term by its modified version (\ref{eq:mod:drift}):
\begin{equation}\label{eq:mod:driftOU}
b_\rho(x)=\left\{\begin{array}{ll}
-\alpha x+\beta &\mbox{if}\ -\rho\le x\le L,\\
\alpha \rho+\beta-\alpha (x+\rho)e^{x+\rho}&\mbox{if}\ x<-\rho.
\end{array}
\right.
\end{equation}
The modified $\gamma$ satisfies $\gamma_\rho(x)=\gamma(x)$ for $x\in[-\rho,L]$ and
\begin{equation}\label{eq:mod:gammaOU}
\gamma_\rho(x)=\frac{1}{2}(\alpha \rho+\beta-\alpha (x+\rho)e^{x+\rho})^2-\frac{\alpha}{2} (1+x+\rho)e^{x+\rho}\quad \mbox{for}\ x<-\rho.
\end{equation}
The function $\gamma_\rho$ is now positive on the whole interval $]-\infty,L]$ and admits the following upper-bound: 
\[
\kappa=\frac{1}{2}\Big(\alpha\rho+\beta+\frac{\alpha}{e}\Big)^2+\frac{\alpha}{2e^2}.
\]
To obtain such an expression, it suffices to upper-bound separately both main terms of \eqref{eq:mod:gammaOU}. We can therefore apply Algorithm (A1) or Algorithm (A2) in order to simulate the approximated first-passage time $\tau_L^\rho$.

Figure \ref{Fig:mod} presents a comparison between the histograms of the hitting time distribution  ($10\,000$ independent simulations and $\rho=5$) with the approximation of the hitting time density obtained via the numerical algorithm proposed in \cite{Buonocore-1987}. We can notice that the histogram and the numerical \emph{pdf} perfectly fit. Moreover, the choice of different larger values for $\rho$ does not significantly affect the results.

\begin{figure}[h]
\centerline{\includegraphics[width=6cm]{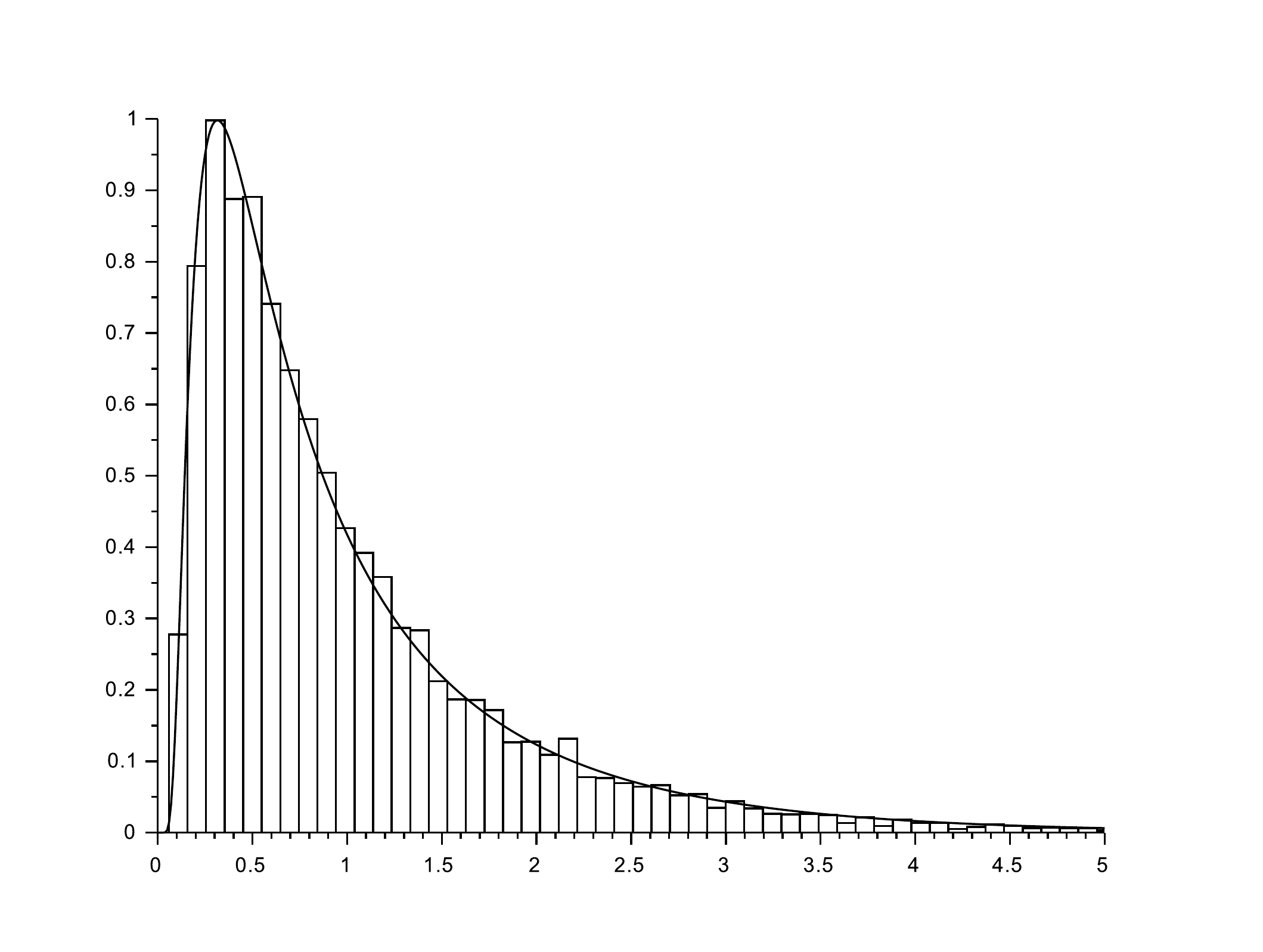}}
\caption{Histogram of the hitting time distribution for 10000 simulations corresponding to the level $L=1$, starting position $X_0=0$ and parameters $a=0.3$, $b=1$ for $\rho=5$ using Algorithm ${\rm (A1)}$ with modified drift, compared with the approximation of the hitting time density obtained via a numerical algorithm.}\label{Fig:mod}
\end{figure}

\FloatBarrier

%\bibliographystyle{plain}
%\bibliography{biblio}

\end{document}